\documentclass[10pt,reqno,oneside,a4paper]{amsart}

\usepackage[table]{xcolor}
\usepackage{subfiles}
\usepackage[english]{babel}
\usepackage[utf8]{inputenc}
\usepackage{amsmath}
\usepackage{amsthm}
\usepackage{amssymb}
\usepackage{amsfonts}
\usepackage{graphicx}
\usepackage{comment}
\usepackage[shortlabels]{enumitem}

\usepackage[top=30truemm,bottom=30truemm,left=30truemm,right=30truemm]{geometry}

\usepackage{adjustbox}
\usepackage{listings}
\usepackage{upgreek}
\usepackage{mathtools}
\usepackage{empheq}
\usepackage{bm}
\usepackage{csquotes}
\usepackage{tikz}
\usetikzlibrary{matrix}
\usepackage{tikz-cd}
\usepackage{setspace}
\usepackage[norelsize]{algorithm2e}
\usepackage{adjustbox}
\usepackage{mathrsfs}
\usepackage{stmaryrd}
\usetikzlibrary{decorations.pathmorphing,shapes}
\usetikzlibrary{decorations.pathreplacing,calligraphy}
\usepackage{rotating}
\usepackage{pdflscape}
\usepackage{subcaption}
\usepackage{afterpage}
\usepackage[all,cmtip]{xy}
\usepackage{tabularray}
\usepackage{stackengine}
\usepackage{caption}
\usetikzlibrary{arrows,calc}
\usepackage{hyperref}
\usepackage{color}
\usepackage{scalerel}
\usepackage{centernot}
\usepackage{float}
\usepackage{tabularx}
\usepackage{empheq}
\usepackage{multirow}
\usepackage{makecell}
\usepackage[normalem]{ulem}
\usepackage{diagbox}
\usetikzlibrary{calc}
\usetikzlibrary{arrows.meta}
\usetikzlibrary{chains}

\allowdisplaybreaks

\makeatletter
\patchcmd{\@setaddresses}{\indent}{\noindent}{}{}
\patchcmd{\@setaddresses}{\indent}{\noindent}{}{}
\patchcmd{\@setaddresses}{\indent}{\noindent}{}{}
\patchcmd{\@setaddresses}{\indent}{\noindent}{}{}
\makeatother

\newcommand{\gap}{\hspace{1pt}}
\newcommand{\ZZ}{\mathbb{Z}}

\renewcommand{\mod}{\text{\normalfont{ mod }}}
\renewcommand{\epsilon}{\varepsilon}

\newcommand{\sfw}{\mathsf{w}}

\newcommand{\chk}{\checkmark}
\DeclareMathOperator{\GKdim}{GKdim}

\DeclareMathOperator{\gldim}{gl.\hspace{-1pt}dim}

\DeclareMathOperator{\idim}{i.\hspace{-1pt}dim}

\DeclareMathOperator{\End}{End}

\DeclareMathOperator{\Ext}{Ext}

\DeclareMathOperator{\hilb}{hilb}

\DeclareMathOperator{\Tr}{Tr}

\DeclareMathOperator{\GL}{GL}
\DeclareMathOperator{\T}{T}
\DeclareMathOperator{\SL}{SL}

\DeclareMathOperator{\hdet}{hdet}

\DeclareMathOperator{\sspan}{span}

\DeclareMathOperator{\lcm}{lcm}

\DeclareMathOperator{\Autgr}{Aut_{gr}}
\DeclareMathOperator{\Aut}{Aut}

\newcommand{\inv}{^{-1}}
\newcommand{\kk}{\Bbbk}
\newcommand{\sqrtw}{\sqrt{\omega}}

\definecolor{darkgreen}{RGB}{55,138,0}
\definecolor{navy}{RGB}{50,50,200}
\definecolor{orn}{RGB}{255,130,0}

\newcommand{\qbin}[3]{\binom{#1}{#2}_{\hspace{-3pt}#3}}

\numberwithin{equation}{section}

\theoremstyle{definition}

\newtheorem{defn}[equation]{Definition}

\newtheorem{example}[equation]{Example}

\theoremstyle{plain}
\newtheorem{thm}[equation]{Theorem}
\newtheorem{prop}[equation]{Proposition}
\newtheorem{lem}[equation]{Lemma}
\newtheorem{cor}[equation]{Corollary}

\theoremstyle{remark}
\newtheorem{rem}[equation]{Remark}

\newcommand\restr[2]{{
  \left.\kern-\nulldelimiterspace 
  #1 
  \right|_{#2} 
  }}

\newcommand\circledmark[1][red!50]{%
  \tikz[baseline=(A.base)]{
    \node[draw,circle,inner sep=0.8ex] (A) {};
    \node[scale=0.9] at (A) {$\checkmark$};
  }%
}

\newcommand*\circled[1]{\tikz[baseline=(char.base)]{
  \node[shape=circle,draw,inner sep=0.01ex] (char) {#1};}}

\numberwithin{equation}{section}

\title{Actions of Taft Algebras on Noetherian Down-Up Algebras}
\author{Simon Crawford, Jason Gaddis, and Robert Won}
\address{(Crawford) The University of Manchester, Alan Turing Building, Oxford Road, Manchester, M13 9PL, United Kingdom}
\email{simon.crawford@manchester.ac.uk}

\address{(Gaddis) Department of Mathematics, 
Miami University, Oxford, Ohio 45056, USA} 
\email{gaddisj@miamioh.edu}

\address{(Won) Department of Mathematics,
The George Washington University, Washington, DC 20052, USA}
\email{robertwon@gwu.edu}

\date{\today}

\begin{document}

\begin{abstract}
We consider actions of Taft algebras on noetherian graded down-up algebras. We classify all such actions and determine properties of the corresponding invariant rings $A^T$. We identify precisely when $A^T$ is commutative, when it is Artin--Schelter regular, and give sufficient conditions for it to be Artin--Schelter Gorenstein. Our results show that many results and conjectures in the literature concerning actions of semisimple Hopf algebras on Artin--Schelter regular algebras can fail when the semisimple hypothesis is omitted.
\end{abstract}

\maketitle


\section{Introduction}

Throughout, we let $\kk$ denote an algebraically closed field of characteristic $0$. All algebras will be associative $\kk$-algebras. 

If $G$ is a finite subgroup of $\GL(n,\Bbbk)$ and $A = \Bbbk[x_1, \dots, x_n]$ is a polynomial ring, then the study of invariants of the action of $G$ on $A$ is a deep and beautiful subject with connections to combinatorics, algebraic geometry, and representation theory. In particular, there are many classical results which describe properties of the invariant ring $A^G$ in terms of properties of $G$. For example, the Shephard--Todd--Chevalley Theorem states that $A^G$ is a polynomial ring if and only if $G$ is generated by quasi-reflections, while Watanabe's Theorem states that, if $G$ contains no quasi-reflections, then $A^G$ is Gorenstein if and only if $G \leqslant \SL(n,\Bbbk)$.

In the past three decades, there has been a strong research effort to try to generalise many of these classical results to a noncommutative setting. The approach that many authors have taken, and which we shall also follow, is to replace the polynomial ring by an \emph{Artin--Schelter regular algebra}, which may be viewed as a ``noncommutative polynomial ring''. Additionally, noncommutative algebras often have ``quantum symmetries'' which commutative algebras lack \cite{EtingofWalton}, so one should allow for actions by finite-dimensional Hopf algebras $H$, rather than just group algebras. For such a Hopf action on an algebra, there exists a suitable notion of an invariant ring, which we denote $A^H$.

This perspective has proved to be fruitful, particularly in the case where one additionally assumes that the Hopf algebra $H$ is semisimple. For example, Kirkman, Kuzmanovich, and Zhang have established an analogue of the Shephard--Todd--Chevalley Theorem for skew polynomial rings \cite{cst}, and an analogue of Watanabe's Theorem for actions of Hopf algebras on Artin--Schelter Gorenstein algebras \cite{kkz}. Many other results have satisfactory generalisations to this noncommutative setting; we refer the interested reader to the survey of Kirkman \cite{Ki} for a thorough overview.


However, much less is known when we omit the semisimple hypothesis from $H$, which many general results in the area require. For example, Molien's Theorem determines the Hilbert series of the invariant ring $A^H$ using representation theory, but this result is false in general if $H$ is not assumed to be semisimple.

In \cite{CWZ}, the authors studied examples of non-semisimple Hopf algebras on Artin--Schelter regular algebras and provided a list of differences between the semisimple and non-semisimple setting \cite[Observation 4.1]{CWZ}. However, these examples exist only in positive characteristic, and so it is unclear if the novel behaviour which they exhibit is also influenced by the characteristic of the field, as is the case in the modular representation theory of groups.

Accordingly, in this paper, we study a certain family of actions of non-semisimple Hopf algebras on Artin--Schelter regular algebras where $\Bbbk$ has characteristic $0$. In particular, we are interested in how properties of the corresponding invariant rings differ from the semisimple setting. The non-semisimple Hopf algebras which we consider are the \emph{Taft algebras}, which were originally defined by Taft \cite{taft} when studying the order of the antipode of a Hopf algebra. These algebras depend on a single parameter $n \geqslant 2$, and have presentation 
\begin{align*}
T_n = \kk\langle x,g \mid g^n-1, \; x^n, \; gx-\omega xg\rangle,
\end{align*}
where $\omega \in \Bbbk$ is a primitive $n$th root of unity. Actions of Taft algebras and their generalisations have been studied on families of AS regular algebras previously. In \cite{GWY}, the authors showed that the invariant ring of a Taft algebra acting on a quantum plane is always a commutative polynomial ring \cite[Lemma 2.7 (1)]{GWY}. In \cite{CG}, actions of \emph{generalized} Taft algebras on quantum planes were considered: here, the invariant rings are always commutative, and are either a polynomial ring or the coordinate ring of a well-understood singularity \cite[Proposition 5.6]{CG}. Other work in this direction include actions on finite-dimensional algebras \cite{MSch,cline}, quantum generalised Weyl algebras \cite{GW2}, path algebras of quivers \cite{KO,KW}, and preprojective algebras \cite{GO}, although these algebras are not AS regular.

Our focus will be actions of Taft algebras on \emph{down-up algebras}. These algebras were originally defined by Benkart and Roby \cite{BRdu}, who were interested in studying their representation theory. We will restrict attention to the subclass of down-up algebras which are noetherian and graded, denoted $A(\alpha,\beta)$ with $\alpha,\beta \in \kk$, $\beta \neq 0$. These are $\kk$-algebras which are generated by two elements subject to two cubic relations, and are AS regular algebras of global dimension 3 \cite{ASch,ATV1}. Kirkman, Kuzmanovich, and Zhang showed that $A(\alpha,\beta)$ is \emph{rigid} in the sense that $A^G$ is not isomorphic to $A$ for any nontrivial subgroup $G$ of automorphisms \cite{KKZ6}. A similar result holds for group coactions; see \cite{CKZ1}.

Our first result classifies all possible inner faithful homogeneous actions of Taft algebras on noetherian-down up algebras:

\begin{thm}[Theorem \ref{thm:summary}] \label{thm:IntroThm1}
Let $n \geqslant 2$ and suppose that $T_n$ acts inner faithfully and homogeneously on a down-up algebra $A(\alpha,\beta)$. Then we have the following possibilities:
\begin{enumerate}[{\normalfont (1)},topsep=1pt,itemsep=1pt,leftmargin=30pt]
\item For some $0 \leqslant k \leqslant n-1$ and $q \in \Bbbk^\times$,
\begin{align*}
g = 
\begin{pmatrix}
\omega^{k+1} & 0 \\ 0 & \omega^k
\end{pmatrix}, 
\quad 
x = \begin{pmatrix}
0 & q \\ 0 & 0
\end{pmatrix},
\quad
\alpha = \omega^{-(k+1)} (1 + \sqrt{\omega}), \quad \beta = - \omega^{-2(k+1)} \sqrt{\omega},
\end{align*}
where $\sqrt{\omega}$ denotes a choice of either one of the square roots of $\omega$; or
\item For some $0 \leqslant k \leqslant n-1$ and $r \in \Bbbk^\times$,
\begin{gather*}
g = 
\begin{pmatrix}
\omega^{k} & 0 \\ 0 & \omega^{k+1}
\end{pmatrix}, 
\quad 
x = \begin{pmatrix}
0 & 0 \\ r & 0
\end{pmatrix},
\quad
\alpha = \omega^{k+1} (1 + \sqrt{\omega^{-1}}), \quad \beta = - \omega^{2(k+1)} \sqrt{\omega^{-1}},
\end{gather*}
where $\sqrt{\omega\inv}$ denotes a choice of either one of the square roots of $\omega\inv$.
\end{enumerate}
Conversely, each of the above parameter choices indeed give rise to an inner faithful action of $T_n$ on an appropriate down-up algebra.
\end{thm}

For the remainder of this introduction, let $A = A(\alpha,\beta)$ and $T = T_n$ be a pair satisfying the conditions of the above theorem. We then seek to determine properties of the invariant rings $A^T$. By Remark \ref{rem:Conjugation}, it will suffice to consider only case (1) from Theorem \ref{thm:IntroThm1}. As a first step, we give a presentation for the subring $A^x$ of $A$ consisting of $x$-invariant elements. It turns out that the structure of $A^x$ depends crucially on whether $\sqrtw$ has order $n$ or order $2n$.

\begin{thm} \label{thm:IntroThm2} 
Suppose that $A$ and $T$ are as in Theorem \ref{thm:IntroThm1} \emph{(1)}.
\begin{enumerate}[{\normalfont (1)},topsep=1pt,itemsep=1pt,leftmargin=30pt]
\item \emph{(Proposition \ref{prop:PresentationForAx})} If $\sqrtw$ has order $n$, then $A^x$ is a skew polynomial ring, generated by $u, z \coloneqq vu-\omega^{-(k+1)}uv,$ and $v^n$. 
\item \emph{Proposition \ref{prop:QuickProofxInvariants})} If $\sqrtw$ has order $2n$, then $A^x$ is generated by $u, z, v^{2n}$ and $x^{n-1} \cdot v^{2n-1}$, and is the factor of a four-dimensional AS regular algebra by a homogeneous regular normal element.
\end{enumerate}
\end{thm}

We are then able to calculate the invariant ring $A^T$ as the subring of $A^x$ consisting of $g$-invariant elements. This allows us to exploit the well-understood properties of actions of finite groups on Artin--Schelter Gorenstein rings. Again, the order of the root of unity $\sqrtw$ has a substantial effect on the properties of $A^T$.

\begin{thm} \label{thm:IntroThm3}
Suppose that $A$ and $T$ are as in Theorem \ref{thm:IntroThm1} \emph{(1)}, and that $\sqrtw$ has order $n$.
\begin{enumerate}[{\normalfont (1)},topsep=1pt,itemsep=1pt,leftmargin=30pt]
\item \emph{(Lemma \ref{lem:CommutativeInvariants})} $A^T$ is commutative.
\item \emph{(Proposition \ref{prop:NonIntermediateInvariants})} $A^T$ is a polynomial ring if and only if $(k+1)(2k+1) \equiv 0 \mod n$. Otherwise, it is the coordinate ring of the product of affine $1$-space with a (possibly non-Gorenstein) type $\mathbb{A}$ singularity.
\item \emph{(Corollary \ref{cor:GorensteinCorollary})} $A^T$ is AS Gorenstein if and only if 
\begin{align*}
 (k+1) \gcd(2k+1,n) + (2k+1) \gcd(k+1,n) \equiv 0 \mod n.
\end{align*}
\end{enumerate}
\end{thm}

\begin{thm} \label{thm:IntroThm4}
Suppose that $A$ and $T$ are as in Theorem \ref{thm:IntroThm1} \emph{(1)}, and that $\sqrtw$ has order $2n$.
\begin{enumerate}[{\normalfont (1)},topsep=1pt,itemsep=1pt,leftmargin=30pt]
\item \emph{(Lemma \ref{lem:NotCommutative})} $A^T$ is not commutative;
\item \emph{(Lemma \ref{lem:NotRegular})} $A^T$ is not AS regular; and
\item \emph{(Theorem \ref{thm:WhenGorenstein})} $A^T$ is AS Gorenstein if $n$ and $k$ satisfy the following relationships:
\begin{enumerate}
\item[\emph{(a)}] $n \geqslant 2$ and $k = n-1$;
\item[\emph{(b)}] $4k+3 \equiv 0 \mod n$ (which happens only if $n$ is odd);
\item[\emph{(c)}] $n \geqslant 3$ and $4k+2 \equiv 0 \mod n$.
\end{enumerate}
\end{enumerate}
\end{thm}

The previous two theorems show that $A^T$ can exhibit a wide variety of behaviours, often quite different from the semisimple setting. In particular, down-up algebras are rigid with respect to actions of Taft algebras, but are not ``strongly rigid'', in the sense that it is possible for $A^T$ to be Artin--Schelter regular, which is new behaviour. Additionally, there is an example where the homological determinant of the action is trivial but $A^T$ is not Artin--Schelter Gorenstein (Example \ref{ex:n2k0Example}), which shows that the semisimple hypothesis is essential in \cite[Theorem 3.3]{gourmet}. Moreover, Example \ref{ex:RegularInvariants} exhibits an action with trivial homological determinant for which the so-called \emph{Auslander map} is not an isomorphism. This shows that one also requires the semisimple hypothesis in \cite[Conjecture 0.2]{ckwz1}.



\mbox{} \\
\noindent\textbf{Organisation of this paper.} In Section 2, we recall some important definitions and results. We classify all possible inner faithful, homogeneous actions of Taft algebras on down-up algebras in Section 3. In Section 4, we establish some useful identities for subsequent calculations which, in Section 5, allows us to determine the subring of $x$-invariant elements. Finally, in Section 6 we use our understanding of the $x$-invariants to understand the full subring of invariants, and some of its properties.

\mbox{} \\
\noindent\textbf{Acknowledgements.} S. Crawford thanks the Heilbronn Institute for Mathematical Research for their financial support during this work. R. Won was partially supported by Simons Foundation grant \#961085. We also thank Benny Liber whose undergraduate research project with J. Gaddis motivated some of the ideas herein.


\section{Background and Preliminaries}

In this section we recall some definitions and preliminary results which are necessary for our analysis.

Let $H$ be a Hopf algebra with counit $\epsilon$, comultiplication $\Delta$, and antipode $S$.
We say $H$ \emph{acts on} an algebra $A$ (from the left)
if $A$ is a left $H$-module, $h \cdot 1_A = \epsilon(A) 1_A$, and for all $h \in H$ and $a,b \in A$,
$h \cdot (ab) = \sum (h_1 \cdot a)(h_2 \cdot b)$. Alternatively, in this setting we say that $A$ is a left $H$-module algebra. If $A$ is a $\ZZ$-graded algebra, we say that the action of $H$ is \emph{homogeneous} if $A_i$ is an $H$-module for each $i \in \ZZ$. If $A$ is an $H$-module algebra, we define the \emph{invariant ring} $A^H$ to be the following subring of $A$:
\begin{align*}
A^H \coloneqq \{ a \in A \mid h \cdot a = \varepsilon(h) a \text{ for all } h \in H \}.
\end{align*}


The action of $H$ is said to be \emph{inner faithful} if there 
does not exist a nonzero Hopf ideal $I$ such that $I \cdot A = 0$. 
When such an ideal exists, then the Hopf algebra $H/I$ acts naturally on $A$. In particular, if $H=\Bbbk G$ is a group algebra, then an action by $H$ is inner faithful if and only if it is faithful.
Thus, the impetus for studying inner faithful actions is that they do not factor through the actions of a ``smaller" Hopf algebra. 

We will be interested in a specific family of Hopf algebras defined below.

\begin{defn}\label{defn.taft}
Fix an integer $n \geqslant 2$ and let $\omega$ be a primitive $n$th root of unity. 
The \emph{Taft algebra} of dimension $n^2$ is the algebra
\[ T_n = \kk\langle x,g \mid g^n-1, \; x^n, \; gx-\omega xg\rangle.\]
\end{defn}
\noindent There is a Hopf structure on $T_n$ defined by declaring $g$ to be grouplike and 
$x$ to be $(g,1)$-skew primitive.
Explicitly, the counit $\epsilon$, comultiplication $\Delta$, and antipode $S$ 
are defined on generators as follows:
\begin{align*}
&\varepsilon(g) = 1, &&\qquad \Delta(g) = g \otimes g, &&\qquad S(g) = g^{n-1} \\
&\varepsilon(x) = 0, &&\qquad \Delta(x) = g \otimes x + x \otimes 1, &&\qquad S(x) = - g^{n-1} x.
\end{align*}
The group of grouplikes of $T_n$ is $C_n = \langle g \rangle$. 
The Taft algebras are non-semisimple, noncommutative, and noncocommutative.

It is known that $\langle x \rangle$ is a Hopf ideal of $T_n$ and $T_n/ \langle x \rangle \cong \Bbbk C_n$, where $C_n$ is the cyclic group of order $n$ generated by $g$.
Therefore, to ensure that the action of $T_n$ is inner faithful, we want to exclude the case where $x$ acts trivially. More formally, we have the result below.

\begin{lem}[{\cite[Lemma 2.5]{KW}}] \label{lem:innerfaithful}
The action of $T_n$ on a $\Bbbk$-algebra is inner faithful if and only if the action of $g$ is given by an automorphism of order $n$ 
and the action of $x$ is nontrivial.
\end{lem}

We do not recall the general definition of a down-up algebra,
as we will only be interested in the following subclass.

\begin{defn}\label{defn.du}
Let $\alpha,\beta \in \kk$.
The \emph{(graded) down-up algebra} with parameters $(\alpha, \beta)$ is the algebra
with presentation
\[A(\alpha,\beta) = \frac{\kk \langle u,v \rangle}{\left\langle
\begin{array}{c}
v^2u - \alpha vuv - \beta uv^2 \\
vu^2 - \alpha uvu - \beta u^2 v
\end{array}
\right\rangle
}.\]
\end{defn}

By \cite[Main Theorem]{KMP}, $A = A(\alpha,\beta)$ is noetherian if and only if $\beta \neq 0$.
We will assume this to be the case throughout. It is also well-known that $A$ is a domain of global dimension and Gelfand--Kirillov (GK) dimension 3. Moreover, for any $\lambda \in \Bbbk$, $A$ has PBW basis
\begin{align*}
\{ u^i (vu-\lambda uv)^j v^k \mid i,j,k \geqslant 0 \}
\end{align*}
and, consequently, $A$ has Hilbert series 
\begin{align*}
\hilb_A (t) = \frac{1}{(1-t)^2(1-t^2)}.
\end{align*}

The \emph{characteristic polynomial} of $A(\alpha,\beta)$ is the polynomial $t^2-\alpha t - \beta$. Let $\gamma_1$, $\gamma_2$ be the roots of the characteristic polynomial. Then $A$ is PI (i.e. satisfies a polynomial identity) if and only if $\gamma_1$ and $\gamma_2$ are roots of unity (see, e.g., \cite[Proposition 6.2.1]{kulkarni}).
Moreover, by \cite[Lemma 7.1]{KKZ6}, the elements $\Omega_i=vu-\gamma_i uv$, $i=1,2$, satisfy
\begin{align}\label{eq.normalcomm}
u\,\Omega_1 = \gamma_2\inv \Omega_1 \, u, \quad
v\,\Omega_1 = \gamma_2 \Omega_1 \, v, \quad
u\,\Omega_2 = \gamma_1\inv \Omega_2 \, u, \quad
v\,\Omega_2 = \gamma_1 \Omega_2 \,v.
\end{align}
\indent When performing computations for Taft actions on down-up algebras, many formulae that arise involve \emph{Gaussian binomial coefficients}, defined as follows:

\begin{defn}
For non-negative integers $m$ and $r$, and $q \in \Bbbk$, define
\begin{align*}
\qbin{m}{r}{q} \coloneqq \frac{(1-q^m)(1-q^{m-1}) \cdots (1-q^{m-r+1})}{(1-q)(1-q^2) \cdots (1-q^r)},
\end{align*}
with the convention that $\qbin{m}{0}{\;q} = 1$.
\end{defn}

It is straightforward to check that, if $m < r$, then $\qbin{m}{r}{\;q} = 0$. Moreover, if $q$ is an $n$th root of unity, then $\qbin{m+n}{r}{\;q} = \qbin{m}{r}{\;q}$. In this case, if $m$ is divisible by $n$ then $\qbin{m}{r}{\;q} = 0$.

Down-up algebras are examples of \emph{Artin--Schelter regular} algebras, which should be thought of as noncommutative analogues of polynomial rings. We define these below, as well as the weaker notion of an \emph{Artin--Schelter Gorenstein} algebra.

\begin{defn} \label{ASregdef}
Let $A$ be a connected graded $\Bbbk$-algebra and write $\Bbbk = A/A_{\geqslant 1}$ for the trivial module. We say that $A$ is \emph{Artin--Schelter Gorenstein} (or AS Gorenstein) \emph{of dimension $d$} if:
\begin{enumerate}[{\normalfont (1)},topsep=1pt,itemsep=1pt,leftmargin=30pt]
\item $\idim {}_A A = \idim A_A = d < \infty$, and 
\item There is an isomorphism of graded right $A$-modules
\begin{align*}
\Ext^i(_{A} \Bbbk,{}_A A) \cong \left \{
\begin{array}{cl}
0 & \text{if } i \neq d \\
\Bbbk[\ell]_A & \text{if } i = d
\end{array}
\right .
\end{align*}
for some integer $\ell$, and a symmetric condition holds for $\Ext^i(\Bbbk_{A},A_A)$, with the same integer $\ell$. We call $\ell$ the \emph{Gorenstein parameter} of $A$.
\end{enumerate}
If, moreover
\begin{enumerate}[{\normalfont (1)},topsep=1pt,itemsep=1pt,leftmargin=30pt]
\item[(3)] $\gldim A = d$, and
\item[(4)] $A$ has finite GK dimension,
\end{enumerate}
then we say that $A$ is \emph{Artin--Schelter regular} (or AS regular) \emph{of dimension $d$}.
\end{defn}

Since we are interested in classifying the inner faithful actions of Taft algebras on down-up algebras, it will be helpful to have a description of the graded automorphism group of $A(\alpha,\beta)$.
As usual, we will restrict attention to inner faithful, homogeneous (left) actions. Let $\T(n,\Bbbk)$ denote the subgroup of $\GL(n,\Bbbk)$ consisting of diagonal matrices, and let $S_2$ denote the subgroup of order $2$ of $\GL(2,\Bbbk)$ generated by the unique nontrivial permutation matrix.

\begin{lem}[{\cite[Proposition 1.1]{KK}}] \label{lem:AutGroup}
The automorphism group of $A(\alpha,\beta)$ is as follows:
\begin{align*}
\Aut_{\mathrm{gr}} A(\alpha,\beta) = \left \{
\begin{array}{ll}
\GL(2,\Bbbk) & \text{if } (\alpha,\beta) \in \{ (0,1), (2,-1) \}, \\
\T(2,\Bbbk) \rtimes S_2 & \text{if } \beta = -1, \alpha \neq 2, \\
\T(2,\Bbbk) & \text{otherwise}.
\end{array}
\right.
\end{align*}
\end{lem}

When $A$ is a connected graded algebra and $H$ is a semisimple Hopf algebra, properties of the invariant ring $A^H$ are well-understood. However, if we remove the semisimple hypothesis from $H$ then this is no longer the case. One of the motivations for the work in this paper is to try to better understand what types of behaviour are possible.

One particular result that no longer holds in this setting is Molien's Theorem, which allows one to calculate the Hilbert series of an invariant ring $A^H$ when $A$ is connected graded and $H$ is a semisimple Hopf algebra. We recall this result in the specific case where $H$ is a group algebra, which we will require later on. We first require a definition:

\begin{defn}
Suppose that $A$ is a connected graded algebra and $g \in \Autgr(A)$. The \emph{trace} of $g$ on $A$ is the formal power series
\begin{align*}
\Tr_A(g,t) \coloneqq \sum_{i \geqslant 0} \Tr(\restr{g}{A_i}) t^i \in \Bbbk \llbracket t \rrbracket, 
\end{align*}
where $\Tr(-)$ is the usual trace of a linear map.

If, moreover, $A$ is AS Gorenstein of dimension $d$ and Gorenstein parameter $\ell$, and $\Tr_A(g,t)$ can be written as a rational function, then we can expand it as a power series in $\Bbbk \llbracket t^{-1} \rrbracket$:
\begin{align*}
\Tr_A(g,t) = (-1)^n c^{-1} t^{-\ell} + \text{lower order terms}.
\end{align*}
The \emph{homological determinant} of $g$ acting on $A$ is defined to be $\hdet_A(g) = c$. This gives rise to a group homomorphism
\begin{align*}
\hdet_A : G \to \Bbbk^\times.
\end{align*}
If, in addition, $A$ is AS regular with Hilbert series
\begin{align*}
\hilb_A(t) = \frac{1}{(1-t)^n p(t)}
\end{align*}
where $p(1) \neq 0$, then we say that $g$ is a \emph{quasi-reflection} if
\begin{align*}
\Tr_A(g,t) = \frac{1}{(1-t)^{n-1} q(t)}
\end{align*}
where $q(1) \neq 0$.
\end{defn}

\begin{rem}
The homological determinant is typically defined using local cohomology, and can be defined for semisimple Hopf algebras as well, but the definition above will be sufficient for our purposes. We note that $\Tr_A(g,t)$ is a rational function when, in particular, $A$ is (a quotient of an) AS regular algebra or is PI \cite[Lemma 4.7]{kkz}.
\end{rem}

\begin{thm}[Molien's Theorem]
Suppose that $G$ is a finite group acting on a connected graded algebra $A$. Then the Hilbert series of $A^G$ can be calculated as follows:
\begin{align*}
\hilb_{A^G}(t) = \frac{1}{n} \sum_{g \in G} \Tr_A(g,t).
\end{align*}
\end{thm}

The homological determinant plays an important role in detecting when an invariant ring as AS Gorenstein:

\begin{thm}[{\cite[Theorem 3.3]{gourmet}}] \label{thm:TrivialHdet}
Suppose that $A$ is noetherian and AS Gorenstein, and that $G$ is a finite subgroup of $\Autgr(A)$. If $\hdet_A(g) = 1$ for each $g \in G$, then $A^G$ is AS Gorenstein.
\end{thm}

When $\hdet_A(g) = 1$ for all $g \in G$, we say that $G$ acts on $A$ with \emph{trivial} homological determinant. While we do not define the homological determinant in the wider generality of Hopf actions, we mention that a general semisimple Hopf algebra $H$ is said to act with trivial homological determinant if $\hdet_A = \varepsilon$ as maps of $\Bbbk$-algebras.

Another way to determine when a connected graded ring is AS Gorenstein is via \emph{Stanley's Theorem}, which was originally proved for commutative rings in \cite[Theorem 4.4]{stanley}. We state a noncommutative version of this result which is suitable for our purposes.

\begin{thm}[Stanley's Theorem, {\cite[Theorem 6.2]{gourmet}}] \label{thm:Stanley}
Let $A$ be an AS Cohen--Macaulay algebra which is a domain and is PI. Then $A$ is AS Gorenstein if and only if the equation of rational functions
\begin{align*}
\hilb_A(t^{-1}) = \pm t^{m} \hilb_A(t)
\end{align*}
holds for some integer $m$.
\end{thm}

We will not define the notion of an AS Cohen--Macaulay algebra here. However, we note that, if $A$ is the invariant ring of a finite group acting on an AS Gorenstein ring, then it is AS Cohen--Macaulay by \cite[Lemma 3.1]{gourmet}.


\section{Classification of Actions}

We begin by classifying all possible actions of a Taft algebra on a down-up algebra. Suppose that $T_n$ acts inner faithfully and homogeneously on $A(\alpha,\beta)$ for some $n \geqslant 2$ and $\alpha,\beta \in \Bbbk$ which are yet to be determined. Abusing notation, write
\begin{align*}
g = \begin{pmatrix} a & b \\ c & d \end{pmatrix}, \qquad x = \begin{pmatrix} p & q \\ r & s \end{pmatrix}
\end{align*}
for the matrices representing the actions of $g$ and $x$, for some $a,b,c,d,p,q,r,s \in \Bbbk$. To determine the values that these parameters can take, we need to ensure that $g$ and $x$ satisfy the defining relations of $T_n$, that the ideal of relations defining $A(\alpha,\beta)$ is invariant under the actions of $g$ and $x$, and that Lemma \ref{lem:innerfaithful} is satisfied. In particular, since $x$ is a $2 \times 2$ matrix, the relation $x^n = 0$ implies $x^2=0$, so that
\begin{gather}
\begin{pmatrix}
p^2 + qr & q(p+s) \\
r(p+s) & s^2 + qr
\end{pmatrix}
= 0, \label{x2}
\end{gather}

From the perspective of invariant theory, it suffices to determine our representation of $H$ up to conjugation by an element of $\Autgr(A)$. By Lemma \ref{lem:AutGroup}, if $(\alpha,\beta) = (0,1)$ or $(\alpha,\beta) = (2,-1)$, by conjugating by a suitable element of $\Autgr(A) = \GL(2,\Bbbk)$, we may assume that $g$ is diagonal. In light of Lemma \ref{lem:AutGroup}, the only other possibility for $g$ is for it to be antidiagonal, which only occurs when $\beta = -1$. In this case, if $\alpha \neq 2$, then it is not possible to diagonalise the actions of $g$.

We first treat the case where $g$ is diagonalisable:

\begin{lem} \label{lem:ExamplesOfActions}
Suppose that $T_n$ acts inner faithfully and homogeneously on some $A(\alpha,\beta)$, where the representation of $T_n$ is given by
\begin{align*}
g = 
\begin{pmatrix}
a & 0 \\ 0 & d
\end{pmatrix}, 
\quad 
x = \begin{pmatrix}
p & q \\ r & s
\end{pmatrix}
\end{align*}
for some $a,d \in \Bbbk^\times$ and $p,q,r,s \in \Bbbk$. Then either:
\begin{enumerate}[{\normalfont (1)},topsep=1pt,itemsep=1pt,leftmargin=30pt]
\item $a=\omega^{k+1}$, $d = \omega^k$ for some $0 \leqslant k \leqslant n-1$, $q \neq 0$, $p=r=s=0$, and 
\begin{align*}
\alpha = \omega^{-(k+1)} (1 + \sqrt{\omega}), \quad \beta = - \omega^{-2(k+1)} \sqrt{\omega}
\end{align*}
where $\sqrt{\omega}$ denotes a choice of either one of the square roots of $\omega$; or
\item $a=\omega^{k}$, $d = \omega^{k+1}$ for some $0 \leqslant k \leqslant n-1$, $r \neq 0$, $p=q=s=0$, and 
\begin{align*}
\alpha = \omega^{k+1} (1 + \sqrt{\omega^{-1}}), \quad \beta = - \omega^{2(k+1)} \sqrt{\omega^{-1}}.
\end{align*}
where $\sqrt{\omega\inv}$ denotes a choice of either one of the square roots of $\omega\inv$. 
\end{enumerate}
Conversely, each of the above parameter choices indeed give rise to an inner faithful action of $T_n$ on an appropriate down-up algebra.
\end{lem}

\begin{proof}
We first note that the relation $gx = \omega xg$ tells us that \begin{gather}
\begin{pmatrix}
(1-\omega)ap & q(a-\omega d) \\
r(d-\omega a) & (1-\omega) ds
\end{pmatrix}
= 0. \label{gx-wxg}
\end{gather}
As $\omega \neq 1$ and $a \neq 0 \neq d$, we deduce that $p=0=s$. By $\eqref{x2}$, we then have $qr=0$ so, to ensure that the action is inner faithful, precisely one of $q$ and $r$ must be nonzero.

Suppose first that $q \neq 0$ and $r = 0$. The top-right entry of \eqref{gx-wxg} implies $q(a-\omega d) = 0$, so that $a = \omega d$. Viewing the variables $u$ and $v$ as elements of the free algebra $\Bbbk\langle u, v\rangle$, using $x \cdot u = 0$ we obtain
\begin{align*}
x \cdot (v u^2) &= (x \cdot v) u^2 = qu^3, \\
x \cdot (uvu) &= (g \cdot u)(x \cdot v) u = q\omega d u^2, \\
x \cdot (u^2 v) &= (g \cdot u) (g \cdot u) (x \cdot v) = q\omega^2 d^2 u^3.
\end{align*}
Therefore,
\begin{align*}
x \cdot (v u^2 - \alpha uvu - \beta u^2 v) = q(1 - \alpha \omega d -\beta \omega^2 d^2)u^3
\end{align*}
and hence we must have
\begin{align}
1 - \alpha \omega d -\beta \omega^2 d^2 = 0. \label{PIeqn}
\end{align}
Similarly, we have
\begin{align*}
x \cdot (v^2 u) &= (g \cdot v) (x \cdot v) u + (x \cdot v) v u = q(dvu^2 + uvu), \\
x \cdot (vuv) &= (g \cdot v) (g \cdot u) (x \cdot v) + (x \cdot v) uv = q(\omega d^2 vu^2 + u^2v), \\
x \cdot (uv^2) &= (g \cdot u) (g \cdot v) (x \cdot v) + (g \cdot u) (x \cdot v) v = q( \omega d^2 uvu + \omega d u^2 v),
\end{align*}
and so
\begin{align}
x \cdot (v^2 u - \alpha vuv - \beta u v^2) 
= q(d - \omega \alpha d^2) vu^2 - q(\omega \beta d^2 - 1) uvu - q(\alpha+\omega \beta d) u^2v. \label{xOnRel1}
\end{align}
Comparing coefficients of $vu^2$ and $uvu$ with those in the first down-up relation shows that
\begin{align}
\alpha q (d-\omega \alpha d^2) = q \omega(\beta d^2 - 1). \label{2ndEqn}
\end{align}
Substituting $\alpha = (\omega d)^{-1}(1- \beta \omega^2 d^2)$ from \eqref{PIeqn} into \eqref{2ndEqn} gives $\omega^3 \beta^2 d^4 - 1 = 0$, and hence
\begin{align*}
\beta = - (\omega d)^{-2} \sqrt{\omega},
\end{align*}
where we write $\sqrtw$ for one of the two possible choices for a square root of $\omega$. Substituting this into \eqref{PIeqn} then gives
\begin{align*}
\alpha = (\omega d)^{-1}(1 + \sqrt{\omega}).
\end{align*}
Now, we must have $d = \omega^k$ for some $0 \leqslant k \leqslant n-1$ to ensure that $g^n=1$. Since $a = \omega d = \omega^{k+1}$, and $k$ and $k+1$ are coprime, it follows that $g$ has order exactly $n$. Finally, it remains to check that these choices give a well-defined action of $T_n$ on $A(\alpha, \beta)$. This is true by construction, with the only exception being that we must verify that $x$ maps the first down-up relation into the ideal of relations defining $A(\alpha,\beta)$. By \eqref{xOnRel1}, this happens if, and only if, $q \beta d (1-\omega \alpha d) = q (\alpha + \omega \beta d)$, and it is straightforward to check that this equation is satisfied for our parameter choices. This corresponds to option (1) in the statement of the lemma.

The analysis when $r \neq 0$ and $q = 0$ is similar and hence omitted; it gives rise to option (2) in the statement.
\end{proof}

It remains to consider the case where $g$ is antidiagonal:

\begin{lem}
There are no inner faithful, homogeneous Taft actions on $A(\alpha,\beta)$ when $\beta = -1$ and $\alpha \neq 2$, where $g$ has the form
\begin{align*}
g = \begin{pmatrix}
0 & b \\ c & 0 
\end{pmatrix}.
\end{align*}
\end{lem}
\begin{proof}
The relation $gx = \omega xg$ now tells us that
\begin{gather*}
\begin{pmatrix}
br-\omega cq & b(s-\omega p) \\
c(p-\omega s) & cq-\omega br
\end{pmatrix}
= 0.
\end{gather*}
Direct calculation gives
\begin{align*}
x \cdot (v^2 u) &= (g \cdot v) (g \cdot v) (x \cdot u) + (g \cdot v) (x \cdot v) u + (x \cdot v) vu \\
&= b(bp+q) u^3 + 0 \cdot v^3 + s v^2 u + 0 \cdot uv^2 + \text{other terms} \\
x \cdot (vuv) &= (g \cdot v) (g \cdot u) (x \cdot v) + (g \cdot v) (x \cdot u) v + (x \cdot v) uv \\
&= 0 \cdot u^3 + 0 \cdot v^3 + 0 \cdot v^2 u + b(cs+r) uv^2 + \text{other terms} \\
x \cdot (uv^2) &= (g \cdot u) (g \cdot v) (x \cdot v) + (g \cdot u) (x \cdot v) v + (x \cdot u) v^2 \\
&= 0 \cdot u^3 + (cs+r) v^3 + 0 \cdot v^2 u + p uv^2 + \text{other terms}
\end{align*}
It follows that
\begin{align*}
x \cdot (v^2u - \alpha vuv - \beta u v^2) = b(bp + q)u^3 + (cs + r)v^3 + s v^2 u + (p - \alpha b r - \alpha b c s) u v^2 + \text{other terms}.
\end{align*}
To ensure that the coefficients of $u^3$ and $v^3$ vanish, we require $r = -cs$ and $q = -bp$. Since the coefficients of $v^2 u$ and $uv^2$ must be equal, this then implies that $p = s$. The top-right entry of \eqref{x2} now tells us that $-2cp^2 = 0$. However, we cannot have $c=0$, otherwise $g$ is singular, nor can we have $p=0$, else $x=0$. This contradiction tells us that no such action is possible.
\end{proof}

This exhausts all possibilities, so we summarise our findings below:

\begin{thm} \label{thm:summary}
Let $n \geqslant 2$ and suppose that $T_n$ acts inner faithfully and homogeneously on a down-up algebra $A(\alpha,\beta)$. Then we have the following possibilities:
\begin{enumerate}[{\normalfont (1)},topsep=1pt,itemsep=1pt,leftmargin=30pt]
\item For some $0 \leqslant k \leqslant n-1$ and $q \in \Bbbk^\times$,
\begin{align*}
g = 
\begin{pmatrix}
\omega^{k+1} & 0 \\ 0 & \omega^k
\end{pmatrix}, 
\quad 
x = \begin{pmatrix}
0 & q \\ 0 & 0
\end{pmatrix},
\quad
\alpha = \omega^{-(k+1)} (1 + \sqrt{\omega}), \quad \beta = - \omega^{-2(k+1)} \sqrt{\omega},
\end{align*}
where $\sqrt{\omega}$ denotes a choice of either one of the square roots of $\omega$; or
\item For some $0 \leqslant k \leqslant n-1$ and $r \in \Bbbk^\times$,
\begin{gather*}
g = 
\begin{pmatrix}
\omega^{k} & 0 \\ 0 & \omega^{k+1}
\end{pmatrix}, 
\quad 
x = \begin{pmatrix}
0 & 0 \\ r & 0
\end{pmatrix},
\quad
\alpha = \omega^{k+1} (1 + \sqrt{\omega^{-1}}), \quad \beta = - \omega^{2(k+1)} \sqrt{\omega^{-1}},
\end{gather*}
where $\sqrt{\omega\inv}$ denotes a choice of either one of the square roots of $\omega\inv$.
\end{enumerate}
Conversely, each of the above parameter choices indeed give rise to an inner faithful action of $T_n$ on an appropriate down-up algebra.
\end{thm}

\begin{cor}
The algebras $A(\alpha,\beta)$ in Theorem \ref{thm:summary} are PI.
\end{cor}
\begin{proof}
Setting $t = (\omega d)^{-1}$ in \eqref{PIeqn} gives the polynomial
\begin{align}
t^2 - \alpha t - \beta, \label{eqn:PIpoly}
\end{align}
which has a root $t_1 = (\omega d)^{-1} = \omega^{-(k+1)}$. The other root $t_2$ satisfies $t_1 t_2 = \beta = \omega^{-2(k+1)} \sqrtw$, and hence $t_2 = \omega^{-(k+1)} \sqrtw$. Therefore both roots of \eqref{eqn:PIpoly} are roots of unity, so $A(\alpha,\beta)$ is PI. This covers case (1) from Theorem \ref{thm:summary}, and the analysis for case (2) is similar.
\end{proof}

\begin{rem}
We are not aware of an inner faithful, homogeneous action of a Taft algebra on an algebra $A$ where $A$ is not PI.
\end{rem}

\begin{cor} \label{cor:NoActions}
There are no inner faithful, homogeneous actions of a Taft algebra on $A(0,1)$ or $A(2,-1)$.
\end{cor}
\begin{proof}
In Theorem \ref{thm:summary}, it is not possible for the parameter $\alpha$ to attain the values $0$ or $2$.
\end{proof}

\begin{rem} \label{rem:Conjugation}
In light of Corollary \ref{cor:NoActions}, Theorem \ref{thm:summary} actually gives all possible homogeneous, inner faithful actions of a Taft algebra on a down-up algebra, not just up to conjugation.

If we allow ourselves to classify the actions up to conjugation, then the actions in (1) and (2) of Theorem \ref{cor:NoActions} above are equivalent, in the following sense. If we conjugate the representation of $T_n$ in (2) by the matrix $\begin{psmallmatrix} 0 & 1 \\ 1 & 0 \end{psmallmatrix}$ then we have
\begin{align*}
g \mapsto 
\begin{pmatrix}
\omega^{k+1} & 0 \\ 0 & \omega^k
\end{pmatrix}, 
\qquad
x \mapsto 
\begin{pmatrix}
0 & r \\ 0 & 0
\end{pmatrix},
\end{align*}
and this has the additional effect of swapping $u$ and $v$ in $A$. Accordingly, the relations on $A(\alpha,\beta)$ become
\begin{align*}
u^2 v - \alpha uvu - \beta v u^2 = -\beta\bigg( vu^2 - \frac{-\alpha}{\beta} uvu - \frac{1}{\beta} u^2 v \bigg), \\
u v^2 - \alpha vuv - \beta v^2 u = -\beta\bigg( v^2u - \frac{-\alpha}{\beta} vuv - \frac{1}{\beta} u v^2 \bigg),
\end{align*}
so that we now have an action on $A\big( -\frac{\alpha}{\beta}, \frac{1}{\beta} \big)$. In particular, if we begin with one of the actions in part (2), then 
\begin{align*}
-\frac{\alpha}{\beta} =  \frac{\omega^{k+1} (1 + \sqrt{\omega^{-1}})}{\omega^{2(k+1)} \sqrt{\omega^{-1}}} = \omega^{-(k+1)}(1+\sqrt{\omega}), \qquad \frac{1}{\beta} = \frac{1}{\omega^{2(k+1)} \sqrt{\omega^{-1}}} = -\omega^{-2(k+1)} \sqrt{\omega} 
\end{align*}
which is of the form given in (1). Accordingly, it suffices to consider only case (1) when analysing properties of the invariant ring $A(\alpha,\beta)^{T_n}$.

Additionally, conjugating by the matrix $\begin{psmallmatrix} p & 0 \\ 0 & 1 \end{psmallmatrix}$, which is an automorphism of any down-up algebra, puts the matrix of $x$ from (1) into Jordan normal form, while leaving the matrix of $g$ and the relations in $A(\alpha,\beta)$ unchanged, so we may assume $p=1$ when analysing $A(\alpha,\beta)^{T_n}$.
\end{rem}

We can restate Theorem \ref{thm:summary} by fixing the down-up algebra $A(\alpha, \beta)$ and determining which Taft algebras act inner faithfully and homogeneously on it. Recall that, for such a down-up algebra, many of its properties are determined the roots of the polynomial $f(t) =  t^2 - \alpha t - \beta$. Call these roots $\gamma_1$ and $\gamma_2$ (which need not be distinct). Note that $\alpha = \gamma_1 + \gamma_2$, $\beta = -\gamma_1\gamma_2$, and $A(\alpha, \beta)$ is PI if and only if both $\gamma_1$ and $\gamma_2$ are roots of unity.

\begin{cor} \label{cor.downuptranslation}
Let $A(\alpha, \beta)$ be a down-up algebra where $(\alpha, \beta) \neq (2, -1), (0,1)$. There is an inner faithful, homogeneous action of some Taft algebra on $A(\alpha, \beta)$ if and only if $\gamma_1$ and $\gamma_2$ are roots of unity, and either $\gamma_1$ or $\gamma_2$ is some power of  $\gamma_1^{-2}\gamma_2^{2}$. In this case, the only Taft algebra which  acts is $T_n$, where $n$ is the multiplicative order of $\gamma_1^{-2}\gamma_2^{2}$. For each $q \in \Bbbk^\times$, there are two actions of $T_n$ on $A(\alpha, \beta)$: either
\begin{align*}
g = 
\begin{pmatrix}
\gamma_1\inv & 0 \\ 0 & \gamma_1 \gamma_2^{-2}
\end{pmatrix}, 
\quad 
x = \begin{pmatrix}
0 & q \\ 0 & 0
\end{pmatrix},
\end{align*}
or
\begin{align*}
g = 
\begin{pmatrix}
\gamma_1\inv \gamma_2^{2} & 0 \\ 0 & \gamma_1 
\end{pmatrix}, 
\quad 
x = \begin{pmatrix}
0 & 0 \\ q & 0
\end{pmatrix},
\end{align*}
where (if necessary) we have relabeled so that $\gamma_1$ is the power of $\gamma_1^{-2} \gamma_2^2$.

In particular, if $A(\alpha, \beta)$ is not PI, then no Taft algebra acts on it inner faithfully and homogeneously.
\end{cor}
\begin{proof}
Note that $\alpha$ and $\beta$ determine the roots of the polynomial $f(t) = t^2 - \alpha t - \beta$ and vice versa. By Theorem~\ref{thm:summary}, if $A(\alpha, \beta)$ admits an inner faithful homogeneous action of some Taft algebra, then we must have
\[
\alpha = \zeta^{-(k+1)}(1 + \sqrt{\zeta}) \quad \text{and} \quad \beta = - \zeta^{-2(k+1)}\sqrt{\zeta}
\]
for some root of unity $\zeta$ and some integer $0 \leqslant k \leqslant n - 1$, where $\sqrt{\zeta}$ denotes a choice of either one of the square roots of $\zeta$. Then the roots of the polynomial $f(t)$ are given by $\zeta^{-(k+1)}$ and $\zeta^{-(k+1)}\sqrt{\zeta}$. In particular, both roots of $f(t)$ are roots of unity and so $A(\alpha, \beta)$ must be PI.
It is clear that in all of the actions given in Theorem~\ref{thm:summary}, one of the roots of $f(t)$ is some power of the square of the quotient of the roots. 

Conversely, suppose that $\gamma_1$ and $\gamma_2$ are roots of unity. If $\gamma_1$ is some power of $\gamma_1^{-2} \gamma_2^{2}$, then let $\omega = \gamma_1^{-2} \gamma_2^{2}$ and choose $0 \leqslant k \leqslant n-1$ such that $\gamma_1 = \omega^{-(k+1)}$. Then we have $\gamma_2^2 = \gamma_1^2 \omega = \omega^{-2(k+1)}\omega$ and so $\gamma_2 = \omega^{-(k+1)}\sqrt{\omega}$ for some choice of $\sqrt{\omega}$. Otherwise, if $\gamma_2$ is some power of $\gamma_1^{-2} \gamma_2^{2}$, then set $\omega = \gamma_1^{2} \gamma_2^{-2}$ and choose $0 \leqslant k \leqslant n - 1$ such that $\gamma_2 = \omega^{-(k+1)}$ so $\gamma_1 = \omega^{-(k+1)} \sqrt{\omega}$. Then by Theorem~\ref{thm:summary}, there exists an action of $T_n$ on $A(\alpha, \beta)$, where $n$ is the order of $\omega = \gamma_1^{-2} \gamma_2^{2}$.

For each fixed down-up algebra of this form, without loss of generality, relabel the roots of $f(t)$ so that $\gamma_1$ is some power of $\gamma_1^{-2}\gamma_2^2$. Then by the paragraph above, we see that $\omega^{k+1} = \gamma_1\inv$ and $\omega^k = \omega^{k+1}/\omega = \gamma_1 \gamma_2^{-2}$ and so Theorem~\ref{thm:summary} (1) gives the action
\begin{align*}
g = 
\begin{pmatrix}
\gamma_1\inv & 0 \\ 0 & \gamma_1 \gamma_2^{-2}
\end{pmatrix}, 
\quad 
x = \begin{pmatrix}
0 & q \\ 0 & 0
\end{pmatrix}
\end{align*}
for each $q \in \kk^\times$. The same down-up algebra $A(\alpha, \beta)$ appears in part (2) of Theorem~\ref{thm:summary}, acted on by the Taft algebra $T_n$ with parameter $\omega\inv$. Hence, Theorem~\ref{thm:summary} (2) gives the action 
\begin{align*}
g = 
\begin{pmatrix}
\gamma_1\inv \gamma_2^{2} & 0 \\ 0 & \gamma_1 
\end{pmatrix}, 
\quad 
x = \begin{pmatrix}
0 & 0 \\ q & 0
\end{pmatrix}
\end{align*}
for each $q \in \kk^\times$. Each down-up algebra appears exactly once in (1) and once in (2) for each choice of $q \in \kk^\times$, which completes the proof.
\end{proof}

We end this section by considering the homological determinant of the actions in Theorem \ref{thm:summary}. By \cite[Theorem 3.2]{crawfordsuperpotentials}, if $H$ is a semisimple Hopf algebra acting on a \emph{derivation-quotient algebra} with associated \emph{twisted superpotential} $\sfw$, then the homological determinant satisfies
\begin{align}
h \cdot \sfw = \hdet_A(h) \sfw \label{eqn:SuperpotentialHdet}
\end{align}
for all $h \in H$. The proof of this result uses \cite[Lemma 5.10]{kkz}, the proof of which does not require $H$ to be semisimple (c.f. \cite[Theorem 2.1]{quantumbinary}). It follows that \eqref{eqn:SuperpotentialHdet} can be used to calculate the homological determinant of the action of any finite-dimensional Hopf algebra on a derivation-quotient algebra. 

Down-up algebras are examples of derivation-quotient algebras, with associated twisted superpotential
\begin{align*}
\sfw = uv^2u - \alpha uvuv - \beta u^2v^2 - \beta^{-1} v^2u^2 + \alpha \beta^{-1} vuvu + v u^2 v,
\end{align*}
viewed as an element in the free algebra $\Bbbk\langle u,v \rangle$. A lengthy computation then shows that, for the actions under consideration in this paper, we have $x \cdot \sfw = 0$. Since $g \in \Autgr(A)$, the usual definition of the homological determinant applies, and shows that $\hdet_A(g) = \omega^{4k+2} = \det(g)^2$ (alternatively, one can apply \eqref{eqn:SuperpotentialHdet} to determine this value). 

\begin{lem} \label{lem:hdet}
Suppose that $T=T_n$ acts on $A=A(\alpha,\beta)$ as in Theorem \ref{thm:summary}. Then
\begin{align*}
\hdet_A (g) = \omega^{4k+2} \quad \text{and} \quad \hdet_A (x) = 0.
\end{align*}
In particular, the homological determinant is trivial if and only if $4k+2 \equiv 0 \mod n$.
\end{lem}


\section{Useful Identities}

In this section, we record various identities that will be useful for computations later in the paper. Many of these identities will depend on certain Gaussian binomial coefficients.

For the remainder of this paper, we write $T = T_n$ for a Taft algebra, depending on some integer $n \geqslant 2$, which acts on a down-up algebra $A = A(\alpha,\beta)$, where both $A$ and the action of $T$ on $A$ depend on some parameter $0 \leqslant k \leqslant n-1$, as in Theorem \ref{thm:summary}. In light of Remark \ref{rem:Conjugation}, it suffices to restrict attention to case (1) of Theorem \ref{thm:summary}, so we will assume that $g$ and $x$ are represented by the following matrices:
\begin{align*}
g = \begin{pmatrix}
\omega^{k+1} & 0 \\ 0 & \omega^k
\end{pmatrix}, \quad
x = \begin{pmatrix}
0 & 1 \\ 0 & 0
\end{pmatrix}.
\end{align*}

In addition, we will henceforth write $z \in A$ for the element 
\begin{align*}
z \coloneqq vu - \omega^{-(k+1)} uv,
\end{align*}
which clearly depends on our choice of action. In particular, we will use the following PBW basis for $A$ for the remainder of this paper:
\begin{align}
\{ u^i z^j v^\ell \mid i,j,\ell \geqslant 0 \}. \label{eqn:PBWbasis}
\end{align}

The next lemma follows directly from \eqref{eq.normalcomm}.

\begin{lem} \label{lem:znormal}
The element $z$ is normal and satisfies
\begin{align*}
zu = \omega^{-(k+1)} \sqrt{\omega} \; uz, \quad \text{and} \quad vz = \omega^{-(k+1)} \sqrt{\omega} \; zv.
\end{align*}
\end{lem}

We have explicitly defined the action of $T$ on the degree $1$ piece of $A$, with its action on higher degree pieces following from the fact that $A$ is a $T$-module algebra. It will be useful to record the action of $x$ on the basis \eqref{eqn:PBWbasis}. 

\begin{lem} \label{lem:AlwaysInvariant}
The elements $u$ and $z$ are $x$-invariant for all $n$ and $k$. In particular, every element in the subring of $A$ generated by $u$ and $z$ is $x$-invariant.
\end{lem}
\begin{proof}
It is clear that $u$ is invariant under $x$, while direct calculation gives
\begin{align*}
x \cdot z &= (g \cdot v)(x \cdot u) + (x \cdot v)u - \omega^{-(k+1)} (g \cdot u)(x \cdot v) - \omega^{-(k+1)} (x\cdot u)v \\
&= u^2 - \omega^{-(k+1)} \omega^{k+1} u^2 = 0. \qedhere
\end{align*}
\end{proof}

\begin{lem} \label{lem:xOnBasis}
On a basis element $u^i z^j v^\ell$, we have
\begin{align*}
x \cdot (u^i z^j v^\ell) = g \cdot (u^i z^j) \; x \cdot v^\ell = \omega^{i(k+1) + j(2k+1)} u^i z^j \gap x \cdot v^\ell.
\end{align*}
\end{lem}
\begin{proof}
Using Lemma \ref{lem:AlwaysInvariant}, we obtain
\begin{gather*}
x \cdot (u^i z^j v^\ell) = g \cdot (u^i z^j) x \cdot v^\ell + x \cdot (u^i z^j) \; v^\ell = g \cdot (u^i z^j) \; x \cdot v^\ell = \omega^{i(k+1) + j(2k+1)} u^i z^j \gap x \cdot v^\ell. \qedhere
\end{gather*}
\end{proof}

It remains to give a formula for the action of $x$ on powers of $v$. We define the following notation, which we will use frequently in this paper:
\begin{align} \label{eqn:mulambda}
\lambda_m \coloneqq \qbin{m}{1}{\omega^{-1}}, \qquad \mu_m \coloneqq \omega^k \qbin{m}{2}{\sqrtw^{-1}}.
\end{align}

\begin{lem} \label{lem:xOnvm}
We have
\[
x \cdot v^m = \lambda_m u v^{m-1} + \mu_m z v^{m-2}.  
\]
\end{lem}
\begin{proof}
We prove this result by induction on $m$. The result holds when $m=1$ since, in this case, both sides of the equation in the lemma are equal to $u$.

Now suppose that $m \geqslant 2$ and that the result holds for smaller $m$. We then have
\begin{align*}
x \cdot v^m &= x \cdot (v \hspace{2pt} v^{m-1}) \\
&= (g \cdot v)(x \cdot v^{m-1}) + (x \cdot v) v^{m-1} \\
&= \omega^k v \Bigg( \qbin{m-1}{1}{\omega^{-1}} u v^{m-2} + \omega^k \qbin{m-1}{2}{\sqrtw^{-1}} z v^{m-3} \Bigg) + u v^{m-1} \\
&= \omega^k \qbin{m-1}{1}{\omega\inv} (z + \omega^{-(k+1)} uv) v^{m-2} + \omega^{k-1} \sqrtw \qbin{m-1}{2}{\sqrtw^{-1}} zv^{m-2} + uv^{m-1} \\
&= \Bigg( 1 + \omega\inv  \qbin{m-1}{1}{\omega\inv} \Bigg) uv^{m-1} + \omega^k \Bigg( \qbin{m-1}{1}{\omega\inv} + \sqrtw^{-1} \qbin{m-1}{2}{\sqrtw^{-1}} \Bigg) zv^{m-2} \\
&= \qbin{m}{1}{\omega^{-1}} u v^{m-1} + \omega^k \qbin{m}{2}{\sqrtw^{-1}} z v^{m-2},
\end{align*}
as desired.
\end{proof}


As a corollary, we have the following result which describes the effect of acting by $x$ on various subspaces of $A$:

\begin{lem} \label{lem:UsefulSubspaces}
Fix integers $p,q,r \geqslant 0$. For $0 \leqslant m \leqslant \frac{r}{2}$, define
\begin{align*}
U_m \coloneqq \sspan \{ u^{p+i} z^{q+m-i} v^{r-2m+i} \mid 0 \leqslant i \leqslant m \}.
\end{align*}
Then $x \cdot U_m \subseteq U_{m+1}$.
\end{lem}
\begin{proof}
The result follows from a direct calculation, using Lemmas \ref{lem:xOnBasis} and \ref{lem:xOnvm}:
\begin{align*}
x \cdot \big( &u^{p+m-i} z^{q+i} v^{r-m-i} \big) \\
&= \omega^{(p+m-i)(k+1) + (q+i)(2k+1)} u^{p+m-i} z^{q+i} (x \cdot v^{r-m-i}) \\
&= \omega^{(p+m-i)(k+1) + (q+i)(2k+1)} u^{p+m-i} z^{q+i} \big( \lambda_{r-m-i} u v^{r-m-i-1} + \mu_{r-m-i} z v^{r-m-i-2} \big) \\
&= \omega^{(p+m-i)(k+1) + k(q+i)} \sqrtw^{q+i} \lambda_{r-m-i} u^{p+(m+1)-i} z^{q+i} v^{r-(m+1)-i} \\
&\hspace{30pt}+ \omega^{(p+m-i)(k+1) + (q+i)(2k+1)} \mu_{r-m-i} u^{p+(m+1)-(i+1)} z^{q+(i+1)} v^{r-(m+1)-(i+1)} \\
&\in U_{m+1},
\end{align*}
where we have recorded the coefficients in the above calculation for later use.
\end{proof}

Finally, the following lemma shows how to write the element $v^m u$ in terms of the basis \eqref{eqn:PBWbasis}, and is a translated version of \cite[Lemma 3.8]{Ghberg}. Rather than making the translation explicit, we provide a direct proof in the notation of this paper.

\begin{lem} \label{lem:vmu}
We have
\begin{align*}
v^m u = \omega^{-m(k+1)} u v^m + \omega^{-(m-1)(k+1)} \qbin{m}{1}{\sqrtw} z v^{m-1}.
\end{align*}
\end{lem}
\begin{proof}
We also prove this result by induction on $m$. The case $m=1$ says that
\begin{align*}
vu = \omega^{-(k+1)} uv + z
\end{align*}
which follows from the definition of $z$. When $m=2$, the right-hand side of the expression in the lemma is
\begin{align*}
\omega^{-2(k+1)} uv^2 + \omega^{-(k+1)}(1 + \sqrtw) zv
&= \omega^{-2(k+1)} uv^2 + \omega^{-(k+1)}(1 + \sqrtw) (vu - \omega^{-(k+1)} uv) v \\
&= \omega^{-(k+1)})(1 + \sqrtw) vuv - \omega^{-2(k+1)} \sqrtw \; uv^2 \\
&= \alpha vuv + \beta uv^2,
\end{align*}
which is equal to $v^2 u$ by definition, as required. 

For the inductive step, assume $m \geqslant 3$. We then have
\begin{align*}
v^m u &= v^2 (v^{m-2} u) \\
&= v^2 \left( \omega^{-(m-2)(k+1)} u v^{m-2} + \omega^{-(m-3)(k+1)} \qbin{m-2}{1}{\sqrtw} z v^{m-3} \right) \\
&= \omega^{-(m-2)(k+1)} \bigg( \omega^{-2(k+1)} uv^2 + \omega^{-(k+1)})(1 + \sqrtw) zv \bigg) v^{m-2} \\
&\hspace{140pt} + \omega^{-(m-3)(k+1)} \omega^{-(2k+1)} \qbin{m-2}{1}{\sqrtw} zv^{m-1} \\
&= \omega^{-m(k+1)} uv^m
+ \omega^{-(m-1)(k+1)} \left( (1 + \sqrtw) + \omega \qbin{m-2}{1}{\sqrtw} \right) zv^{m-1} \\
&= \omega^{-m(k+1)} uv^m
+ \omega^{-(m-1)(k+1)} \qbin{m}{1}{\sqrtw} zv^{m-1},
\end{align*}
as desired.
\end{proof}


\section{The subring of \texorpdfstring{$x$}{x}-invariants}

The remainder of this paper is dedicated to determining the invariant ring $A^T$ and some of its properties. Our strategy for computing $A^T$ will be to first determine the elements which are invariant under the action of $x \in T$, which we denote $A^x$, and then to determine which of these are additionally invariant under the action of $g \in T$. Assuming $A^x$ is suitably nice, this will allow us to use tools from the invariant theory of finite groups.

In general, the elements of a ring $R$ which are invariant under the action of a single element of a Hopf algebra need not form a subring of $R$. However, it is well-known that in a Hopf algebra $H$, if $g_1, g_2 \in H$ are grouplike and $h \in H$ is $(g_1, g_2)$-primitive, then $A^h$ is a subring of $A$. In particular, we have the following:

\begin{lem}
Let $A^x$ denote the subset of $A$ consisting of $x$-invariant elements:
\begin{align*}
A^x = \{ a \in A \mid x \cdot a = 0 \}.
\end{align*}
Then $A^x$ is a subring of $A$.
\end{lem}

It will turn out that there is a dichotomy in the behaviour of the invariants, depending on whether $\sqrtw$ has order $n$ or order $2n$. In this section, we consider these two cases separately.

\subsection{The case when \texorpdfstring{$\sqrtw$}{sqrt(w)} has order \texorpdfstring{$n$}{n}}
In this subsection, we make the standing assumption that the root of unity $\sqrtw$ has order $n$, and will not repeat this in the statements of results. This case can only occur when $n$ is odd, and amounts to choosing $\omega^{(n+1)/2}$ as the square root of $\omega$.

We begin by identifying some elements in $A^x$.

\begin{lem} \label{lem:SomeInvariants}
We have $u,z,v^n \in A^x$.
\end{lem}
\begin{proof}
That $u$ and $z$ are $x$-invariant is Lemma \ref{lem:AlwaysInvariant}. To see that $v^n$ is $x$-invariant, we appeal to Lemma \ref{lem:xOnvm}. Indeed, we have
\begin{align*}
x \cdot v^n = \qbin{n}{1}{\omega^{-1}} u v^{n-1} + \omega^k \qbin{n}{2}{\sqrtw^{-1}} z v^{n-2}.  
\end{align*}
Since $\omega^{-1}$ and $\sqrtw^{-1}$ are primitive $n$th roots of unity, both of the Gaussian binomial coefficients in the above expression vanish, as desired.
\end{proof}

In fact, these elements generate $A^x$, and this subring has a nice form:

\begin{prop} \label{prop:QuickProofxInvariants}
The subring $A^x$ is generated by $u,z,$ and $v^n$, and is a skew polynomial ring where
\begin{align*}
v^n u = u v^n, \quad zu = \omega^{-(k+1)} \sqrt{\omega} \; uz, \quad v^n z = zv^n.
\end{align*}
In particular, $A^x$ is AS regular.
\end{prop}
\begin{proof}
We can write any $p \in A$ as $p = \sum_{m=0}^d p_m v^m$ for some 
$p_m$ in the subalgebra generated by $u$ and $z$, which is a skew polynomial ring.
We may assume $p_d \neq 0$. 
Note that each of the $p_m$ is $x$-invariant, by Lemma \ref{lem:SomeInvariants}. Suppose $p$ as above lies in $A^x$. We then have
\begin{align*}
0 = x \cdot p = \sum_{m=0}^d g \cdot p_m \; x \cdot v = \sum_{m=0}^d g \cdot p_m (\lambda_m uv^{m-1} + \mu_m z v^{m-2}) 
\end{align*}
by Lemma \ref{lem:xOnvm}.
In the above, the coefficient of $v^{d-1}$ is $\lambda_d \gap (g \cdot p_d) \gap u$. Since $p_d \neq 0$, this forces $\lambda_d = 0$, and so $d$ must be divisible by $n$, say $d = d' n$ for some integer $d'$. By Lemma \ref{lem:SomeInvariants}, $v^d = (v^n)^{d'}$ is $x$-invariant. Then $p - p_d (v^n)^{d'} \in A^x$ has strictly smaller $v$-degree than $p$, and an induction argument now shows that $p$ can be generated by $u,z,$ and $v^n$.

It is straightforward to check that the given elements (skew-)commute as claimed in the statement of the proposition. Moreover, there are no algebraic relations between $u$, $z$, and $v^{2n}$ because $\{ u^i z^j v^\ell \mid i,j,\ell \geqslant 0 \}$ is a PBW basis for $A$. Therefore $A^x$ is a skew polynomial ring.
\end{proof}

We have the following corollary, which will be helpful when trying to understand the full invariant ring $A^T$. 

\begin{cor}
The Hilbert series of $A^x$ is
\begin{align*}
\hilb_{A^x}(t) = \frac{1}{(1-t)(1-t^2)(1-t^n)}.
\end{align*}
\end{cor}

\subsection{The case when \texorpdfstring{$\sqrtw$}{sqrt(w)} has order \texorpdfstring{$2n$}{2n}}

In this subsection, we make the standing assumption that that the root of unity $\sqrtw$ has order $2n$, and will not repeat this in the statements of results. This is always the case when $n$ is even, and can also occur when we choose the appropriate square root of $\omega$ when $n$ is odd, namely $\sqrtw = - \omega^{(n+1)/2}$.


As in the previous section, we begin by determining some elements of $A$ which are $x$-invariant. The differences in the behaviour of the invariant ring in this subsection, compared with the previous subsection, stem from fact that $v^n$ is no longer $x$-invariant. Instead, the smallest power of $v$ which is $x$-invariant is $v^{2n}$:

\begin{lem} \label{lem:SomeMoreInvariants}
We have $u, z, v^{2n} \in A^x$, while $v^n \notin A^x$.
\end{lem}
\begin{proof}
We only need to establish the claims about $v^{2n}$ and $v^n$. Lemma \ref{lem:xOnvm} tells us that
\begin{align*}
x \cdot v^m = \qbin{m}{1}{\omega^{-1}} u v^{m-1} + \omega^k \qbin{m}{2}{\sqrtw^{-1}} z v^{m-2}.  
\end{align*}
and both of the Gaussian binomial coefficients vanish when $m=2n$, but only the former vanishes when $m=n$.
\end{proof}

Since $v^n$ is not $x$-invariant, there is a possibility that there are elements in degree at most $2n$ which cannot be generated by just $u$ and $z$. Much of this section will be devoted to identifying one such element, and determining its properties. 

For the remainder of this paper, we write
\begin{align*}
a \coloneqq x^{n-1} \cdot v^{2n-1},
\end{align*}
which is clearly $x$-invariant. For future reference, we record how $g$ acts on $a$:
\begin{align}
g \cdot a = g (x^{n-1} \cdot v^{2n-1}) = \omega^{n-1} (x^{n-1} \cdot g v^{2n-1}) = \omega\inv (x^{n-1} \cdot \omega^{(2n-1)k} v^{2n-1}) = \omega^{-(k+1)} a. \label{eqn:gOna}
\end{align}

We claim that $A^x$ is equal to the subalgebra generated by the elements $a$, $u$, $z$, and $v^{2n}$, namely
\begin{align}
B \coloneqq \Bbbk \langle a, u, z, v^{2n} \rangle. \label{eqn:Bsubalgebra}
\end{align}
We first require a lemma:

\begin{lem} \label{lem:LeadingCoeffOfa}
If we write $a = x^{n-1} \cdot v^{2n-1}$ in terms of the PBW basis \eqref{eqn:PBWbasis}, then $u^{n-1} v^n$ appears with nonzero coefficient, and every other term appears with smaller $v$-degree. 
\end{lem}
\begin{proof}
To see this, first set $p=0$, $q=0$, and $r = 2n-1$ in Lemma \ref{lem:UsefulSubspaces}; then $v^{2n-1} \in U_0$ and the spaces
\begin{align*}
U_m \coloneqq \sspan \{ u^{i} z^{m-i} v^{2n-1-2m+i} \mid 0 \leqslant i \leqslant m \}, \qquad 0 \leqslant m \leqslant n-1,
\end{align*}
satisfy $x \cdot U_m \subseteq U_{m+1}$. In particular, this shows that every term in $a$ has $v$-degree at most $n$. Now, if $p \in U_m$, the only term that contributes to the coefficient of $u^{m+1} v^{2n-m-2}$ in $x \cdot p \in U_{m+1}$ is the term arising from $x \cdot u^m v^{2n-m-1}$. Thus
\begin{align}
x^{n-1} \cdot v^{2n-1} &= x^{n-2} \cdot (\lambda_{2n-1} uv^{2n-2} + \text{other terms}) \nonumber \\
&= x^{n-3} \cdot (\lambda_{2n-1} \lambda_{2n-2} \omega^{k+1} u^2v^{2n-3} + \text{other terms}) \nonumber \\
&\hspace{40pt} \vdots \nonumber \\
&= \Bigg( \prod_{i=m}^{n-1} \lambda_{2n-m} \omega^{(m-1)(k+1)} \Bigg) u^{n-1} v^n + \text{other terms}, \label{eqn:UsefulCoeff}
\end{align} 
where the coefficient in \eqref{eqn:UsefulCoeff} is nonzero. 
\end{proof}

We are now in a position to show that $B = A^x$.

\begin{prop} \label{prop:GensOfAx}
Let $B$ be the subalgebra of $A$ generated by
\begin{align*}
u, \quad z, \quad v^{2n}, \quad a \coloneqq x^{n-1} \cdot v^{2n-1},
\end{align*}
as in \eqref{eqn:Bsubalgebra}. Then $B=A^x$.
\end{prop}
\begin{proof}
By Lemma \ref{lem:SomeMoreInvariants}, we have $B \subseteq A^x$, so it remains to establish the reverse inclusion. To this end, let $p \in A^x$. Since $p \in A$, we can write
\begin{align*}
p = \sum_{m=0}^d p_m(u,z) v^m
\end{align*}
for some $p_m = p_m(u,z) = \sum_{i,j} \alpha_{i,j,m} u^i z^j$, where $p_d \neq 0$. Note that if $d \equiv 0 \mod 2n$, then $x \cdot \big(p_d v^d \big) = 0$, and $p - p_d v^d \in A^x$, so we may as well assume $d \not\equiv 0 \mod 2n$.

Since $p \in A^x$, using the notation of \eqref{eqn:mulambda}, we have
\begin{align*}
0 = x \cdot p
&= \sum_{m=0}^d (g \cdot p_m) (x \cdot v^m) \\
&= \sum_{m=0}^d (g \cdot p_m) \left( \lambda_m uv^{m-1} + \mu_m zv^{m-2} \right) \\
&= \sum_{m=0}^d \lambda_m (g \cdot p_m) uv^{m-1} + \mu_m (g \cdot p_m) z v^{m-2}.
\end{align*}
The coefficient of $v^{d-1}$ in $x \cdot p$ is $\lambda_d (g \cdot p_d) u = \qbin{d}{1}{\hspace{2pt}\omega\inv} (g \cdot p_d) u$, which is necessarily equal to $0$. Since $p_d \neq 0$, we must have $\qbin{d}{1}{\hspace{2pt}\omega\inv} = 0$. This forces $d \equiv 0 \mod n$, and since $d \not\equiv 0 \mod 2n$, we must have $d \equiv n \mod 2n$. 

Now, the coefficient of $v^{d-2}$ in $x \cdot p$ is necessarily $0$, so we obtain
\[
\lambda_{d-1} (g \cdot p_{d-1}) u = -\mu_d (g\cdot p_d) z.
\]
Since $d \equiv n \mod 2n$, both $ \qbin{d-1}{1}{\hspace{2pt}\omega\inv}$ and $\binom{d}{2}_{\sqrt{\omega}\inv}$ are nonzero and so, since the right hand side of the above equality has a factor of $u$, it follows that $u$ divides $g \cdot p_d$, and hence also divides $p_d$. 

The same argument holds for each $m$ with $d-n \leqslant m \leqslant d-2$: the coefficient of $v^{m}$ is
\[
\lambda_{m+1} (g \cdot p_{m+1}) u + \mu_{m+2} (g\cdot p_{m+2}) z 
\]
where both $\lambda_{m+1}$ and $\mu_{m+2}$ are nonzero.
Since $p_{m+1} u$ has a factor of $u$, the same must be true of $p_{m+2}$.

It follow that $u^{n-1}$ is a factor of $p_d$, and hence the term in $p$ containing $v^d$ has the form
\begin{align*}
p_d(u,z) v^d = \widetilde{p}_d(u,z) u^{n-1} v^n (v^{2n})^\ell,
\end{align*}
for some polynomial $\widetilde{p}_d(u,z)$ in $u$ and $z$, and some non-negative integer $\ell$. The element $\widetilde{p}_d a (v^{2n})^\ell$ lies in $B$ and, by Lemma \ref{lem:LeadingCoeffOfa}, has the same coefficient of $v^d$ as $p$. It follows that, for a suitably chosen $\gamma \in \Bbbk$, the element $p - \gamma \widetilde{p}_d a (v^{2n})^\ell \in A^x$, has strictly smaller $v$-degree than $p$. An induction argument now shows that $p \in B$.
\end{proof}

Having written down a generating set for $A^x$, our goal is to now give a presentation for this ring. In particular, we will need to determine skew-commutativity relations between our generating set, as well as any other relations among them.

We first give a closed-form for an element which is closely related to $a$.

\begin{prop} \label{prop:CloselyRelatedElement}
We have
\begin{align*}
x^{n-1} \cdot v^{2n-2} = \omega^{\frac{1}{2}n(n-1)(2k+1) + k(n-1)} \prod_{i=1}^{n-1} \qbin{2i}{2}{\sqrtw\inv} \hspace{5pt} z^{n-1}.
\end{align*}
\end{prop}
\begin{proof}
In Lemma \ref{lem:UsefulSubspaces}, set $p=0$, $q=0$, and $r=2n-2$, which tells us that the subspaces
\begin{align*}
U_m \coloneqq \sspan \{ u^i z^{m-i} v^{2n-2-2m+i} \mid 0 \leqslant i \leqslant m \}, \qquad 0 \leqslant m \leqslant n-1,
\end{align*}
satisfy $x \cdot U_m \subseteq U_{m+1}$. Since $v^{2n-2} \in U_0$, we obtain
\begin{align*}
x^{n-1} \cdot v^{2n-2} \in U_{n-1} = \sspan \{ u^{i} z^{n-1-i} v^{i} \mid 0 \leqslant i \leqslant n-1 \}.
\end{align*}
We also know that $x^{n-1} \cdot v^{2n-2}$ is $x$-invariant, so consider an arbitrary element of $U_{n-1} \cap A^x$, say
\begin{align*}
p = \sum_{i=0}^{n-1} \alpha_i u^{i} z^{n-1-i} v^{i}.
\end{align*}
Then
\begin{align*}
0 = x \cdot p &= \sum_{i=0}^{n-1} \alpha_i \gap g \cdot (u^{i} z^{n-1-i}) \gap x \cdot v^{i} \\
&=\sum_{i=1}^{n-1} \alpha_i \gap g \cdot (u^{i} z^{n-1-i}) \gap \Big( \lambda_{i} uv^{i-1} + \mu_{i} z v^{i-2} \Big).
\end{align*}
The coefficient of $v^{n-2}$ is $\alpha_{n-1} \lambda_{n-1} g \cdot (u^{i} z^{n-i-i}) u$ and, since $\lambda_{n-1} \neq 0$, we deduce that $\alpha_{n-1} = 0$. Repeating this argument shows that $\alpha_i = 0$ for $1 \leqslant i \leqslant n-1$, so that $p = \alpha_{0} z^{n-1}$. It follows that $x^{n-1} \cdot v^{2n-2}$ is a scalar multiple of $z^{n-1}$. To determine the value of this scalar, we note that if $p \in U_m$, the only term that contributes to the coefficient of $z^{m+1} v^{2n-2-2(m+1)}$ in $x \cdot p \in U_{m+1}$ is the term arising from $x \cdot (z^m v^{2n-2-2m})$, namely $\omega^{m(2k+1)} \mu_{2n-2-2m} z^{m+1} v^{2n-2-2(m+1)}$. Therefore,
\begin{align*}
x^{n-1} \cdot v^{2n-2} 
&= x^{n-2} \cdot (\mu_{2n-2} z v^{2n-4} + \text{ other terms}) \\
&= x^{n-3} \cdot (\mu_{2n-2} \mu_{2n-4} \omega^{2k+1} z^2 v^{2n-6} + \text{ other terms}) \\
& \hspace{20pt} \vdots \\
&= \left(\prod_{i=1}^{n-1} \omega^{(i-1)(2k+1)} \mu_{2i} \right) z^{n-1} \\
&= \omega^{\frac{1}{2}n(n-1)(2k+1) + k(n-1)} \prod_{i=1}^{n-1} \qbin{2i}{2}{\sqrtw\inv} \hspace{5pt} z^{n-1}. \qedhere
\end{align*}
\end{proof}

With the above result in hand, we are able to determine the skew-commutativity relations between the elements of $B$.

\begin{lem} \label{lem:CommutativityRels}
The element $v^{2m}$ is central in $B$, while the other generators satisfy the relations
\begin{gather*}
zu = \omega^{-(k+1)} \sqrtw \gap\gap uz, \quad az = \omega^{-(k+1)} \sqrtw \gap\gap za, \hspace{5pt} \\
au-ua = -\omega^{\frac{1}{2}n(n-1)(2k+1) + k(n-1)} \sqrtw \prod_{i=1}^{n-1} \qbin{2i}{2}{\sqrtw\inv} \hspace{5pt}  z^n.
\end{gather*}
\end{lem}
\begin{proof}
By Lemma \ref{lem:vmu}, we have
\begin{align*}
v^{2n} u = \omega^{-2n(k+1)} u v^m + \omega^{-(2n-1)(k+1)} \qbin{2n}{1}{\sqrtw} z v^{m-1} = u v^{2n},
\end{align*}
where the second term vanishes since $\sqrtw$ is a primitive $(2n)$th root of unity. Therefore $v^{2n}$ is central in $A$, and hence also in $B$. The fact that $zu = \omega^{-(k+1)} \sqrtw \; uz$ follows from Lemma \ref{lem:znormal}.

For the skew-commutativity relation between $z$ and $a$, we calculate
\begin{align*}
az &= (x^{n-1} \cdot v^{2n-1}) z \\
&= x^{n-1} \cdot (v^{2n-1} z) \\
&= \big( \omega^{-(k+1)} \sqrtw \big)^{2n-1} x^{n-1} \cdot (z v^{2n-1}) \\
&= \big( \omega^{-(k+1)} \sqrtw \big)^{-1} (g^{n-1} \cdot z) (x^{n-1} \cdot v^{2n-1}) \\
&= \big( \omega^{-(k+1)} \sqrtw \big)^{-1} \omega^{-(2k+1)} za \\
&= \omega^{-(k+1)} \sqrtw \gap za.
\end{align*}

For the final relation, we have
\begin{align*}
au-ua &= (x^{n-1} \cdot v^{2n-1}) u - u (x^{n-1} \cdot v^{2n-1}) \\
&= x^{n-1} \cdot (v^{2n-1} u - \omega^{-(n-1)(k+1)} u v^{2n-1}) \\
&= x^{n-1} \cdot \Bigg(-\omega^{-(2n-2)(k+1)} \qbin{2n-1}{1}{\sqrtw} zv^{2n-2}\Bigg) \\
&= -\omega^{2(k+1)} \sqrtw\inv x^{n-1} \cdot ( z v^{2(n-1)}) \\
&= -\omega^{2(k+1)} \sqrtw\inv \omega^{(n-1)(2k+1)} \gap z \gap (x^{n-1} \cdot v^{2(n-1)}) \\
&= -\sqrtw \gap z \gap (x^{n-1} \cdot v^{2(n-1)}).
\end{align*}
The result now follows from Proposition \ref{prop:CloselyRelatedElement}.
\end{proof}

Since $B$ has $4$ generators which satisfy relations similar to those in a quantum polynomial ring and $A^x$ should have GK dimension $3$, one expects the generators of $B$ to satisfy another homogeneous relation and, indeed, this is the case. We first require a technical lemma, whose proof is similar to that of Lemma \ref{lem:UsefulSubspaces}, so we omit it.

\begin{lem}
For $0 \leqslant m \leqslant n-1$, set
\begin{align*}
V_m = \sspan \{ u^{2m+i} z^{n-m-1-i} v^{2n+i} \mid 0 \leqslant i \leqslant n-m-1 \}.
\end{align*}
Then $x \cdot V_m \subseteq V_{m+1}$.
\end{lem}

\begin{lem} \label{lem:TechnicalFora2}
With $V_m$ as above, we have $a v^{2n-1} \in V_0$, and the coefficient of $u^{n-1} v^{3n-1}$ in $a v^{2n-1}$ is
\begin{align}
\omega^{\frac{1}{2}(n-1)(n-2)(k+1)} \prod_{i=m}^{n-1} \qbin{m}{1}{\omega\inv}. \label{eqn:eta}
\end{align}
\end{lem}
\begin{proof}
Consider the subspaces $U_m$ from Lemma \ref{lem:UsefulSubspaces} with $p=0$, $q=0$, and $r=2n-1$, so that
\begin{align*}
U_m \coloneqq \sspan \{ u^i z^{m-i} v^{i-1+2(n-m)} \mid 0 \leqslant i \leqslant m \}, \qquad 0 \leqslant m \leqslant n-1,
\end{align*}
satisfy $x \cdot U_m \subseteq U_{m+1}$. In particular, $U_0 = \Bbbk v^{2n-1}$ and $U_{n-1} v^{2n-1} = V_0$, so that
\begin{align*}
a v^{2n-1} = (x^{n-1} \cdot v^{2n-1}) v^{2n-1} \in U_{n-1} v^{2n-1} = V_0.
\end{align*}
To determine the coefficient of $u^{n-1} v^{3n-1}$ in $a v^{2n-1}$, it suffices to determine the coefficient of $u^{n-1} v^n$ in $a$.
This coefficient was calculated in \eqref{eqn:UsefulCoeff}, and it is straightforward to check that it simplifies to the coefficient given in the statement of the lemma.
\end{proof}

We are now able to write down an algebraic relation between the generators of $B$:

\begin{prop}\label{prop:a2Relation}
We have
\begin{align*}
a^2 = \omega^{2(k+1)} \Bigg( \prod_{m=1}^{n-1} \qbin{m}{1}{\omega^{-1}} \Bigg)^2 u^{2n-2} v^{2n}.
\end{align*}
\end{prop}
\begin{proof}
We have
\begin{align*}
a^2 &= a \; x^{n-1} \cdot v^{2n-1} \\
&= x^{n-1} \cdot \big( (g^{-(n-1)} \cdot a) v^{2n-1} \big) \\
&= \omega^{-(k+1)} x^{n-1} \cdot (a v^{2n-1}).
\end{align*}
Now, $a v^{2n-1} \in V_0$ by Lemma \ref{lem:TechnicalFora2}, and hence $x^{n-1} (a v^{2n-1}) \in V_{n-1} = \sspan \{ u^{2n-2} v^{2n} \}$. To determine the coefficient of $u^{2n-2} v^{2n}$, we only need to keep track of a single term at each step, as in the proof of Lemma \ref{lem:TechnicalFora2}. Writing $\eta$ for the scalar \eqref{eqn:eta}, we have
\begin{align*}
x^{n-1} \cdot (a v^{2n-1})
&= x^{n-1} \cdot (\eta u^{n-1} v^{3n-1} + \text{other terms}) \\
&= \eta \Bigg(\prod_{m=n-1}^{2n-3} \omega^{m(k+1)} \Bigg) \Bigg( \prod_{m=2n+1}^{3n-1} \lambda_m \Bigg) u^{2n-2} v^{2n} \\
&= \eta \gap\gap \omega^{\frac{1}{2}(n-1)(n-4)(k+1)} \Bigg( \prod_{m=1}^{n-1} \qbin{m}{1}{\omega^{-1}} \Bigg) u^{2n-2} v^{2n} \\
&= \omega^{3(k+1)} \Bigg( \prod_{m=1}^{n-1} \qbin{m}{1}{\omega\inv} \Bigg)^2 u^{2n-2} v^{2n}.
\end{align*}
The claim now follows.
\end{proof}

With these results in hand, we can give a presentation for $A^x$. The following result shows that the algebra whose skew-commutativity relations are the same as those satisfied by the generators of $B$ has good properties.

\begin{lem}
Consider the algebra
\begin{align*}
\frac{\Bbbk \langle a,b,c,d \rangle}{
\left \langle
\begin{array}{c}
\begin{array}{ccc}
ad-da, & bd-db, & cd-dc
\end{array} \\
\begin{array}{cc}
ca-\omega^k \sqrt{\omega}\: ac, & bc-\omega^k \sqrt{\omega}\: cb   
\end{array} \\
\begin{array}{c}
ba-ab-c^n
\end{array}
\end{array}
\right \rangle
}
\end{align*}
with grading given by declaring that the respective degrees of $a,b,c,d$ are $2n-1,1,2,2n$. This algebra is a noetherian domain which is AS regular of dimension 4. 
\end{lem}

\begin{proof}
It is straightforward to show that this algebra is the iterated Ore extension
\begin{align*}
\Bbbk[a][c;\sigma][b;\tau,\delta][d]
\end{align*}
where 
\begin{gather*}
\sigma : \Bbbk[a] \to \Bbbk[a], \qquad a \mapsto \omega^k \sqrt{\omega} \\
\tau: \Bbbk[a][c;\sigma] \to \Bbbk[a][c;\sigma], \qquad a \mapsto a, \quad c \mapsto \omega^k \sqrt{\omega}
\end{gather*}
are ring automorphisms, and $\delta : \Bbbk[a][c;\sigma] \to \Bbbk[a][c;\sigma]$ is the $\tau$-derivation satisfying
\begin{gather*}
\delta(a) = c^n, \quad \delta(c) = 0.
\end{gather*}
All of the claims now follow.
\end{proof}

\begin{lem}
Consider the algebra from the preceding lemma. Then the element $f \coloneqq a^2 - b^{2(n-1)}d$ is a normal nonzerodivisor which satisfies
\begin{align*}
af = fa, \quad bf=fb, \quad cf=\omega^{2k+1}fc.
\end{align*}
\end{lem}
\begin{proof}
Upon noting that
\begin{align*}
ba^2 = (ab+c^n)a = a(ab+c^n) + (\omega^k \sqrt{\omega})^n a c^n = a^2b + ac^n - ac^n = a^2 b,
\end{align*}
the first two identities follow immediately. For the last identity, we have
\begin{align*}
cf &= c(a^2 - b^{2(n-1)} d) = (\omega^k \sqrt{\omega})^2 a^2 c - (\omega^k \sqrt{\omega})^{-(2(n-1))} b^{2(n-1)} d c = \omega^{2k+1} (a^2 - b^{2(n-1)} d)c \\
&= \omega^{2k+1} fc.
\end{align*}
Finally, the fact that $f$ is a nonzerodivisor follows from the algebra being a domain.
\end{proof}

Finally, we can give a presentation for $A^x$:

\begin{prop} \label{prop:PresentationForAx}
We have the following presentation for $A^x$:
\begin{align*}
A^x \cong \frac{\Bbbk \langle a,b,c,d \rangle}{
\left \langle
\begin{array}{c}
\begin{array}{ccc}
ad-da, & bd-db, & cd-dc
\end{array} \\
\begin{array}{cc}
ca-\omega^k \sqrt{\omega}\: ac, & bc-\omega^{k} \sqrt{\omega}\: cb   
\end{array} \\
\begin{array}{cc}
ba-ab-c^n, & a^2-b^{2(n-1)} d
\end{array}
\end{array}
\right \rangle
}.
\end{align*}
\end{prop}
\begin{proof}
Call the right-hand side $C$. We define a map $\theta : \Bbbk\langle a,b,c,d \rangle \to A^x$ as follows. First declare $\theta(b) = u$ and $\theta(d) = v^{2n}$. We let $\theta(a) = \gamma x^{n-1} \cdot v^{2n-1}$, where $\gamma \in \Bbbk^\times$ is chosen so as to ensure that $\theta(a^2 - b^{2(n-1)}d) = 0$, which is possible by Proposition \ref{prop:a2Relation}. We then define $\theta(c) = \eta z$ where, in light of Corollary \ref{lem:CommutativityRels}, our choice of $\eta \in \Bbbk^\times$ gives $\theta(ba-ab-c^n) = 0$. Lemma \ref{lem:CommutativityRels} and Proposition \ref{prop:a2Relation} ensure that we obtain a well-defined ring homomorphism $\overline{\theta} : C \to A^x$, which is surjective by Proposition \ref{prop:GensOfAx}.

The result will follow if we can show that $\overline{\theta}$ is injective. First note that $A^x \supseteq \Bbbk\langle u, z, v^{2n} \rangle$ which, using arguments similar to those in Theorem \ref{prop:QuickProofxInvariants}, is a polynomial ring in three variables. It follows that $\GKdim A^x \geqslant 3$. On the other hand, since $C$ is obtained by factoring an algebra of GK dimension $4$ by a regular normal element, we must have $\GKdim C \leqslant 3$ by \cite[Proposition 3.15]{lenagan}. By the same reference again, the kernel of $\overline{\theta}$ must have GK dimension $0$, and hence be finite-dimensional. If the kernel were nonzero, then this would imply the existence of at least one additional relation in $A^x$ of degree at most $4n-3$.

To rule out this possibility, assign bidegrees $(n-1,n), (1,0), (1,1),$ and $(0,2n)$ to the elements $a,b,c,$ and $d$, which are simply the $(u,v)$-bidegrees of their images in $A^x$. Using the relations in $C$, we can write any monomial in $C$ in the form $a^i b^j c^\ell d^m$, and the only pairs of monomials of degree at most $4n-3$ which have the same bidegree are of the form $a b^{i+1} c^j$ and $b^i c^{n+j}$. Any tentative relation must say that these are equal up to scaling, and deleting copies of $b$ and $c$ (since we are in a domain), says that $ab$ is a scalar multiple of $c^n$, which is absurd. It follows that there are no additional relations, so the kernel of $\overline{\theta}$ is trivial, and we obtain the desired isomorphism.
\end{proof}

In particular, using a similar argument as in \cite[Proposition 3.2]{KK} we have the following:

\begin{cor} \label{cor:AxIsGorenstein}
$A^x$ is AS Gorenstein.
\end{cor}
\begin{proof}
Since $A^x$ is a factor of an iterated Ore extension by a regular normal element, it is Auslander--Gorenstein by \cite[Theorem 4.2]{ekstrom} and \cite[3.6 Theorem]{levasseur}. Then, since $A^x$ has finite GK dimension and is connected, it is also AS Gorenstein by \cite[6.3 Theorem]{levasseur}.
\end{proof}


\section{Properties of the invariant ring \texorpdfstring{$A^T$}{A T}}

We end this paper by establishing a few properties of the full invariant rings $A^T$, including when they are commutative, when they are AS regular, and when they are AS Gorenstein. As in the previous section, we need to distinguish between the cases when $\sqrtw$ has order $n$ and when $\sqrtw$ has order $2n$. 


\subsection{The case when \texorpdfstring{$\sqrtw$}{sqrt(w)} has order \texorpdfstring{$n$}{n}}

All results in this subsection will assume that $\sqrtw$ has order $n$. As explained previously, we seek to calculate the invariant ring $A^T$ as $A^T = (A^x)^g$. To this end, we first write down the action of $g$ on the (skew) polynomial ring $A^x = \Bbbk[u,z,v^n]$. Direct calculation shows that the matrix of $g$ on this ring (with the variables in the given order) is
\begin{align}
\begin{pmatrix}
\omega^{k+1} & 0 & 0 \\ 0 & \omega^{2k+1} & 0 \\ 0 & 0 & \phantom{1} 1\phantom{1} \label{eqn:gOnAx}
\end{pmatrix}.
\end{align}
For the remainder of this subsection, let $G$ be the group generated by the matrix \eqref{eqn:gOnAx}.

We have the following basic result which, since $A^x$ is only ``mildly'' noncommutative, is not too surprising.

\begin{lem}\label{lem:CommutativeInvariants}
$A^{T}$ is commutative.
\end{lem}
\begin{proof}
Since the action of $g$ on $A^x$ is diagonal, it suffices to consider which monomials in $u,z,$ and $v^n$ are invariant under $g$. By \eqref{eqn:gOnAx}, a monomial $u^i z^j v^{n\ell} \in A^{x}$ is $g$-invariant if and only if $\omega^{(k+1) i + (2k+1)j} = 1$, which happens if and only if $(k+1) i + (2k + 1)j \equiv 0 \mod n$. Consider any two such monomials $u^i z^j v^{n\ell}$ and $u^{i'} z^{b'} v^{n c'} \in A^{T}$, so that
\begin{gather*}
(k+1)i + (2k+1)j \equiv 0 \mod n \\
(k+1)i' + (2k+1)j' \equiv 0 \mod n .
\end{gather*}
If we multiply the first congruence by $j'$ and the second by $j$ and then subtract, we obtain ${(k+1)(ij' - ji') \equiv 0 \mod n}$. Similarly, multiplying the first by $i'$ and the second by $i$ and subtracting yields $(2k + 1)(i j' - j i') \equiv 0 \mod n$. Since $\gcd(k+1, 2k+1) = 1$, we conclude that $ij' - ji' \equiv 0 \mod n$. Writing $\varepsilon_k = \omega^{-(k+1)} \sqrtw$, which is a primitive $n$th root of unity, we obtain
\begin{align*}
(u^i z^j v^{n\ell})(u^{i'} z^{j'} v^{n \ell'}) &= u^i z^j (u^{i'} z^{j'} v^{n \ell'}) v^{n \ell} \\
&= \varepsilon_k^{-ji'} u^i (u^{i'} z^{j'} v^{n \ell'}) z^j v^{n \ell} \\
&= \varepsilon_k^{ij' - ji'} (u^{i'} z^{j'} v^{n \ell'}) (u^i z^j v^{n \ell}) \\
&= (u^{i'} z^{j'} v^{n \ell'}) (u^i z^j v^{n \ell}),
\end{align*}
and so $A^T$ is commutative.
\end{proof}

Since $G$ is a group acting on the skew polynomial ring $A^x$ by \cite[Theorem 1.1]{cst}, we see that $A^T$ is AS regular if and only if $G$ is generated by quasi-reflections. Indeed, since $A^T$ is commutative, the AS regularity of $A^T$ is equivalent to $A^T$ being a (commutative) polynomial ring in three variables. The quasi-reflections on $A^x$ are also well-understood by \cite{cst}. In particular, since $g$ acts diagonally, the only elements which act as quasi-reflections on $A$ are the (classical) reflections on $A_1$ \cite[Proposition 4.1]{cst}.

It remains to understand the classical reflections in $G$.

\begin{lem} \label{lem:GeneratedByReflections} Let $d = \gcd(k + 1, n)$ and $e = \gcd(2k + 1, n)$.
\begin{enumerate}[{\normalfont (1)},topsep=1pt,itemsep=1pt,wide=0pt]
\item The group $G$ is quasi-reflection-free if and only if $k+1$ and $2k+1$ are both coprime to $n$ if and only if $de = 1$.
\item The group $G$ is generated by quasi-reflections if and only if $(k+1)(2k+1) \equiv 0 \mod n$ if and only if $de = n$.
\item The group $G$ contains, but is not generated by, quasi-reflections if and only if $1 < de < n$.
\end{enumerate}
\end{lem}
\begin{proof}
As discussed above, the matrix giving the action of $g$ on the generators of $A^x$ is 
\begin{align*}
\begin{pmatrix}
\omega^{k+1} & 0 & 0 \\ 0 & \omega^{2k+1} & 0 \\ 0 & 0 & \phantom{1} 1\phantom{1}
\end{pmatrix}.
\end{align*}
By an abuse of notation, we shall also refer to this matrix as $g$. Clearly, for any $m \in \ZZ$, the diagonal entries of $g^m$ are $\omega^{m(k+1)}$, $\omega^{m(2k+1)}$, and $1$. Hence, the reflections in $G$ are exactly the elements $g^m$ such that exactly one of $\omega^{m(k+1)}$ and $\omega^{m(2k+1)}$ is equal to $1$, where $1 \leqslant m < n$.

(1) If both $k+1$ and $2k + 1$ are coprime to $n$, then $\omega^{m(k+1)} = 1$ if and only if $n \mid m$ if and only if $\omega^{m(2k+1)} = 1$. Hence, $G$ is quasi-reflection-free. Conversely, suppose that at least one of $k + 1$ and $2k + 1$ is not coprime to $n$. Without loss of generality, suppose that $d \neq 1$. Note that since $k+1$ and $2k + 1$ are coprime, therefore $d$ is not a divisor of $2k + 1$. Then letting $m = n/d$, we see that $\omega^{m(k+1)} = 1$ while $\omega^{m(2k+1)} \neq 1$, so $g^{n/d}$ is a quasi-reflection and so $G$ is not quasi-reflection-free.

(2) Observe that 
\[
g^{n/d} = \begin{pmatrix}
1 & 0 & 0 \\ 0 & \omega^{n(2k+1)/d} & 0 \\ 0 & 0 & \phantom{1} 1\phantom{1}
\end{pmatrix} \quad \text{ and } \quad g^{n/e} = \begin{pmatrix}
\omega^{n(k+1)/e} & 0 & 0 \\ 0 & 1 & 0 \\ 0 & 0 & \phantom{1} 1\phantom{1}
\end{pmatrix}.
\]
Further observe that every quasi-reflection in $G$ must be a power of $g^{n/d}$ or $g^{n/e}$. Hence, if $G$ is generated by quasi-reflections, then we must actually have $G = \langle g^{n/d}, g^{n/e} \rangle$.

Therefore, $G$ is generated by quasi-reflections if and only if $g = g^{ni/d}g^{nj/e}$ for some $i,j \in \ZZ$. This is true if and only if there exist $i,j \in \ZZ$ such that $\omega^{k+1} = \omega^{nj(k+1)/e}$ and $\omega^{2k+1} = \omega^{ni(2k+1)/d}$. But this happens if and only if $
\gcd(k + 1,n) = \gcd(n(k+1)/e, n)$ and $
\gcd(2k + 1,n) = \gcd(n(2k+1)/d, n)$. Since $\gcd(k+1, n) = d$ and $\gcd(2 k + 1, n) = e$, this happens if and only if $\gcd(d, n) = \gcd(nd/e, n)$ and $\gcd(e, n) = \gcd(ne/d, n)$.

For any integer $m$ and any prime $p$, let $\nu_p(m)$ denote the largest integer $k$ such that $p^k \mid m$. We have that $\gcd(d, n) = \gcd(nd/e, n)$ if and only if $\nu_p(d) < \nu_p(n)$ implies that $\nu_p(e) = \nu_p(n)$ for all primes $p$. Therefore, $G$ is generated by quasi-reflections if and only if  $\nu_p(d) < \nu_p(n)$ implies that $\nu_p(e) = \nu_p(n)$ and $\nu_p(e) < \nu_p(n)$ implies that $\nu_p(d) = \nu_p(n)$. That is, any prime factor of $n$ which does not occur in $d$ must occur with full multiplicity in $e$. But since $k+1$ and $2k + 1$ are coprime, so are $d$ and $e$. Therefore, this condition is equivalent to $de = n$, which is equivalent to $(k + 1)(2k + 1) \equiv 0 \mod n$.

(3) Since we have seen that case (1) occurs if and only if $de = 1$ and case (2) occurs if and only if $de = n$, the remaining case is when $de$ is neither $1$ nor $n$. Since $k+1$ and $2k + 1$ are coprime, this is equivalent to $1 < de < n$.
\end{proof}

\begin{prop} \label{prop:NonIntermediateInvariants} 
Let $d = \gcd(k + 1, n)$ and $e = \gcd(2k + 1, n)$.
\begin{enumerate}[{\normalfont (1)},topsep=1pt,itemsep=1pt,wide=0pt]
\item If $de=1$, so that $k+1$ and $2k+1$ are both coprime to $n$, then $A^T$ is the coordinate ring of the product of affine $1$-space with a (possibly non-Gorenstein) type $\mathbb{A}$ singularity. Explicitly,
\begin{align*}
A^T \cong \Bbbk[x,y]^{\frac{1}{n}(1,r)} [t],
\end{align*}
where $r = (2k+1) (k+1)^{-1} \mod n$ and
\begin{align*}
\frac{1}{n}(1,r) = \left\langle \begin{pmatrix}
\omega & 0 \\ 0 & \omega^r
\end{pmatrix} \right \rangle.
\end{align*}
In particular, $A^T$ is AS Gorenstein if and only if $3k+2 \equiv 0 \mod n$.
\item If $de =n$ (equivalently, $(k+1)(2k+1) \equiv 0 \mod n$), then
\begin{align*}
A^T = \Bbbk[u^{n/d}, z^{n/e},v^n]
\end{align*}
is a polynomial ring.
\item If $1 < de < n$ , then $A^T$ is the coordinate ring of the product of affine $1$-space with a (possibly non-Gorenstein) type $\mathbb{A}$ singularity. Moreover, it is AS Gorenstein if and only if
\begin{align*}
e (k+1) + d (2k+1) \equiv 0 \mod n.
\end{align*}
\end{enumerate}
\end{prop}
\begin{proof} \mbox{}\\
(1) Let $i$ be the inverse of $k+1$, modulo $n$. Then the $i$th power of the matrix \eqref{eqn:gOnAx} is
\begin{align*}
\begin{pmatrix}
\omega & 0 & 0 \\
0 & \omega^r & 0 \\
0 & 0 & 1
\end{pmatrix}
\end{align*}
and, since $i$ and $n$ are coprime, this matrix generates $G$. It follows that, using the notation of \cite[Section 2.2]{crawfordactions},
\begin{align*}
A^T = (A^x)^G = \Bbbk_{\varepsilon_k}[u,z]^{\frac{1}{n}(1,r)}[v^n],
\end{align*}
where $\varepsilon_k = \omega^{-(k+1)} \sqrtw$ and $zu = \varepsilon_k uz$. This ring is necessarily commutative by Lemma \ref{lem:CommutativeInvariants}, so it follows that
\begin{align*}
\Bbbk_{\varepsilon_k}[u,z]^{\frac{1}{n}(1,r)} \cong \Bbbk[x,y]^{\frac{1}{n}(1,r)},
\end{align*}
where $\Bbbk[x,y]$ is a commutative polynomial ring. The exact form of the invariant ring $\Bbbk[x,y]^{\frac{1}{n}(1,r)}$ can be understood by \cite[Theorem 6.5]{crawfordactions}, for example; it is the coordinate ring
of the product of affine $1$-space with a (possibly non-Gorenstein) type $\mathbb{A}$ singularity. By \cite[Theorem 1]{WatanabeII}, this ring is known to be AS Gorenstein if and only if $G$ is a subgroup of $\SL(3,\Bbbk)$, and this happens if and only $3k+2 \equiv 0 \mod n$.

(2) Since $(k+1)(2k+1) \equiv 0 \mod n$, $G$ is generated by quasi-reflections by Lemma \ref{lem:GeneratedByReflections}, and hence $A^T$ has finite global dimension by \cite[Theorem 0.1]{cst}. By Lemma \ref{lem:CommutativeInvariants}, $A^T$ is also commutative, and hence it is a polynomial ring. Clearly $v^n$ is an invariant and, since $\omega^{k+1}$ and $\omega^{2k+1}$ have respective orders $n/d$ and $n/e$, it follows that $u^{n/d}$ and $z^{n/e}$ are also invariants, and one can show that these three elements generate $A^T$.

(3) Finally, suppose $1 < de < n$. We calculate the invariant ring $A^T = (A^x)^G$ via a series of intermediate steps. First note that the matrix representing the action of $g^{n/e}$ on $A^x$ is
\begin{align*}
h_1 \coloneqq
\begin{pmatrix}
\omega^{(k+1) n/e} & 0 & 0 \\
0 & 1 & 0 \\
0 & 0 & 1
\end{pmatrix},
\end{align*}
where $\omega^{(k+1) n/e}$ has order $e$. Therefore $(A^x)^{h_1} = \Bbbk[u^e,z,v^n]$, where $u^e$ and $z$ skew-commute, at worst. Now, the matrix representing the action of $g$ on $(A^x)^{h_1}$ is
\begin{align*}
\begin{pmatrix}
\omega^{e(k+1)} & 0 & 0 \\
0 & \omega^{2k+1} & 0 \\
0 & 0 & 1
\end{pmatrix},
\end{align*}
and hence the matrix of $g^{n/de}$ on $(A^x)^{h_1}$ is
\begin{align*}
h_2 \coloneqq
\begin{pmatrix}
1 & 0 & 0 \\
0 & \omega^{\frac{2k+1}{e} \frac{\gap n \gap}{d}} & 0 \\
0 & 0 & 1
\end{pmatrix},
\end{align*}
where $\omega^{\frac{2k+1}{e} \frac{\gap n \gap}{d}}$ has order $d$. Hence $\big((A^x)^{h_1}\big)^{h_2} = \Bbbk[u^e,z^d,v^n]$. Finally, the matrix of $g$ restricted to this algebra is
\begin{align}
\begin{pmatrix}
\omega^{e(k+1)} & 0 & 0 \\
0 & \omega^{d(2k+1)} & 0 \\
0 & 0 & 1
\end{pmatrix}, \label{eqn:gOnAxh}
\end{align}
which has order $n/de$ since both $\omega^{e(k+1)}$ and $\omega^{d(2k+1)}$ have order $n/de$. It follows that the cyclic group generated by this matrix acts as a quasi-reflection-free group on $\big((A^x)^{h_1}\big)^{h_2} = \Bbbk[u^e,z^d,v^n]$ and so, in principal, one can give an explicit generating set for the (necessarily commutative) invariant ring $A^T = \Bbbk[u^e,z^d,v^n]^g$. In particular, it is the coordinate ring of the product of affine $1$-space with some (possibly non-Gorenstein) type $\mathbb{A}$ singularity. Finally, by \cite[Theorem 1]{WatanabeII}, it is AS Gorenstein if and only if the matrix \eqref{eqn:gOnAxh} has determinant $1$; equivalently,
\begin{align*}
e (k+1) + d(2k+1) \equiv 0 \mod n,
\end{align*}
as claimed.
\end{proof}

\begin{example} \label{ex:RegularInvariants}
Proposition \ref{prop:NonIntermediateInvariants} (2) shows that it is possible for $A^T$ to be AS regular, which is not the case when we replace $T$ by a group algebra or its dual \cite{rigidity,CKZ1}.


In particular, if $n$ is odd and $k=\frac{n-1}{2}$, we obtain
\begin{align*}
A^T = \Bbbk[u^n,z,v^n].
\end{align*}
By the discussion after Corollary \ref{cor.downuptranslation}, in this case the action of $T$ has trivial homological determinant. Moreover, from the PBW basis \eqref{eqn:PBWbasis}, it is clear that $A$ is free as an $A^T$-module: explicitly, it has rank $n^2$ and basis
\begin{align*}
\{ u^i v^\ell \mid 0 \leqslant i,\ell < n \}.
\end{align*}
In particular, there exist elements of $\End_{A^T}(A)$ of negative degree, and so \cite[Conjecture 0.2]{ckwz1} does not hold if we omit the semisimple hypothesis.

Another interesting feature of this example is that, by \cite[Theorem 1.3 (f)]{zhao}, the invariant ring $A^T$ is equal to the centre of $A$. It would be interesting to investigate for which down-up algebras the centre can be obtained as the subring of invariants under the action of some Hopf algebra. 
\end{example}


The condition in part (3) of Proposition \ref{prop:NonIntermediateInvariants} actually holds in all cases, and gives a convenient criterion to determine when $A^T$ is AS Gorenstein:

\begin{cor} \label{cor:GorensteinCorollary}
$A^T$ is AS Gorenstein if and only if
\begin{align}
 (k+1) \gcd(2k+1,n) + (2k+1) \gcd(k+1,n) \equiv 0 \mod n. \label{eqn:GorensteinCondition}
\end{align}
\end{cor}
\begin{proof}
We consider the three cases of Lemma \ref{lem:GeneratedByReflections}. We have already seen that the result holds for case (3), and the result holds for case (1) since, in this setting,
\begin{align*}
 (k+1) \gcd(2k+1,n) + (2k+1) \gcd(k+1,n) = (k+1) + (2k+1) = 3k+2.
\end{align*}
In case (2), $A^T$ is always AS Gorenstein, so it remains to show that if $(k+1)(2k+1) \equiv 0 \mod n$ then \eqref{eqn:GorensteinCondition} holds. If $k+1 = n$ or $2k+1 = n$ then this is clear, so suppose not, in which case $\gcd(k+1,n)$ and $\gcd(2k+1,n)$ lie strictly between $1$ and $n$. Since 
\begin{align*}
k+1 = \gcd(k+1,n) \frac{\lcm(k+1,n)}{n},
\end{align*} 
from $(k+1)(2k+1) \equiv 0 \mod n$ we obtain
\begin{align*}
(2k+1) \gcd(k+1,n) \frac{\lcm(k+1,n)}{n} \equiv 0 \mod n.
\end{align*}
But $\lcm(k+1,n)/n$ and $n$ are coprime, so we deduce that
\begin{align*}
(2k+1) \gcd(k+1,n) \equiv 0 \mod n.
\end{align*}
Similarly, one can show that
\begin{align*}
(k+1) \gcd(2k+1,n) \equiv 0 \mod n,
\end{align*}
from which the claim follows.
\end{proof}

We note that the homological determinant of the action of $T$ on $A$ is trivial precisely when $k=\frac{n-1}{2}$ (Lemma \ref{lem:hdet}), and then $A^T$ is AS Gorenstein by Proposition \ref{prop:NonIntermediateInvariants} (2). In particular, \cite[Theorem 3.3]{gourmet} holds for these examples.

\subsection{The case when \texorpdfstring{$\sqrtw$}{sqrt(w)} has order \texorpdfstring{$2n$}{2n}}
We finally establish some properties of the full invariant rings $A^T$ when $\sqrtw$ has order $2n$, a hypothesis which we will not repeat. We will typically compute this ring as the subring of $g$-invariant elements of the ring given in Proposition \ref{prop:PresentationForAx}. Accordingly, we will need to determine the action of $g$ on the generators of this ring. Since $a,b,c,$ and $d$ are, respectively, scalar multiples of $x^{n-1} \cdot v^{2n-1}, u, z,$ and $v^{2n}$, it is straightforward to check that the matrix of $g$ on this algebra is given by
\begin{align}
\begin{pmatrix}
\omega^{-(k+1)} & 0 & 0 & 0 \\
0 & \omega^{k+1} & 0 & 0 \\
0 & 0 & \omega^{2k+1} & 0 \\
0 & 0 & 0 & 1
\end{pmatrix}. \label{eqn:gOnAx2}
\end{align}
Many properties of $A^T$ can now be determined using standard techniques, since we can view it as the subring of the AS Gorenstein ring $A^x$ consisting of elements which are invariant under the cyclic group generated by $g$. 

In Lemma \ref{lem:CommutativeInvariants} we saw that, when $\sqrtw$ has order $n$, the invariant ring $A^T$ is always commutative. On the other hand, in our current setting we have the complete opposite behaviour:

\begin{lem} \label{lem:NotCommutative}
$A^T$ is not commutative.
\end{lem}
\begin{proof}
Throughout, let $\varepsilon_k = \omega^{-(k+1)} \sqrt{\omega}$, which is a primitive $(2n)$th root of unity, and view $A^x$ via its presentation in Proposition \ref{prop:PresentationForAx}. From this, it is clear that $c^n$ is $g$-invariant. Moreover, the element $b^{2k+1} c^{n-(k+1)}$ is $g$-invariant, since 
\begin{align*}
(2k+1)(k+1) + (n-(k+1))(2k+1) \equiv 0 \mod n,
\end{align*}
where we note that the exponent of $c$ is non-negative, and that the exponent of $b$ is odd. These elements do not commute, since
\begin{align*}
b^{2k+1} c^{n-(k+1)} \cdot c^n = \varepsilon_k^{-n(2k+1)} c^n \cdot b^{2k+1} c^{n-(k+1)} = - c^n \cdot b^{2k+1} c^{n-(k+1)},
\end{align*}
and hence $A^T$ is not commutative.
\end{proof}

We now seek to determine situations when $A^T$ is AS Gorenstein. We have two main tools for this: we can apply Theorem \ref{thm:TrivialHdet} after determining when the homological determinant of $g$ on $A^x$ is trivial, or we can apply Stanley's Theorem. We note that the hypotheses of Stanley's Theorem are always met for the examples of interest to us, since $A^T = (A^x)^g$ is the ring of invariants for a finite group acting on an AS Gorenstein ring, and therefore it is AS Cohen--Macaulay \cite[Lemma 3.1 (1)]{gourmet}. Moreover, it is a subring of the PI domain $A$, and hence also has these properties.

To apply either Theorem \ref{thm:TrivialHdet} or Stanley's Theorem, we need to know the trace of $g$ on $A^x$. 

\begin{lem} \label{lem:TraceLemma}
The trace of $g^m$ on $A^x$ is
\begin{align*}
\Tr_{\restr{g^m}{A^x}}(t) = \frac{1-\omega^{-2m(k+1)} t^{4n-2}}{(1-\omega^{m(k+1)}t)(1-\omega^{m(2k+1)}t^2)(1-\omega^{-m(k+1)}t^{2n-1})(1-t^{2n})}.
\end{align*}
\end{lem}
\begin{proof}
This is a direct application of \cite[Lemma 1.6]{KKZ6}.
\end{proof}

The following example demonstrates an application of these techniques: 

\begin{example} \label{ex:EasyExample}
Let $n \geqslant 2$ and $k=n-1$. Then the matrix of $g$ on $A^x$ is
\begin{align*}
\begin{pmatrix}
1 & 0 & 0 & 0 \\
0 & 1 & 0 & 0 \\
0 & 0 & \omega^{n-1} & 0 \\
0 & 0 & 0 & 1
\end{pmatrix},
\end{align*}
and so a simple computation using Molien's Theorem and Lemma \ref{lem:TraceLemma} gives
\begin{align*}
\hilb_{A^T}(t) = \frac{1-t^{4n-2}}{(1-t)(1-t^{2n-1})(1-t^{2n})^2}.
\end{align*}
Direct computation then gives
\begin{align*}
\hilb_{A^T}(t^{-1}) = -t^{2(n+1)} \hilb_{A^T}(t),
\end{align*}
so Stanley's Theorem tells us that $A^T$ is AS Gorenstein. It is easy to see that the invariant ring is generated by $a,b$ and $d$ (the invariant $c^n$ satisfies $c^n = ba-ab$, so would be redundant). With a bit more work, one can show that $A^T$ is isomorphic to a polynomial ring over the down-up algebra $A(0,1)$ (which is AS regular), modulo the hyperplane relation $a^2 = b^{2(n-1)} d$.
\end{example}

Before determining when $A^T$ is AS Gorenstein, we first investigate whether it is possible for it to have the stronger property of being AS regular. In Proposition \ref{prop:NonIntermediateInvariants} (1) we saw that, when $\sqrtw$ has order $n$, it was possible for the invariant ring $A^T$ to be AS regular. However, when $\sqrtw$ has order $2n$ this never happens:

\begin{lem} \label{lem:NotRegular}
$A^T$ is not AS regular.
\end{lem}
\begin{proof}
Since $A^T = (A^x)^g$ is the ring of invariants of a ring of GK dimension $3$ under the action of a finite group, it has GK dimension $3$. Therefore, if it were AS regular, then it would have $2$ or $3$ generators, by \cite[Proposition 1.1 (1)]{stephenson}. Hence, to prove the result, it suffices to show that a purported generating set of size $3$ cannot generate the entire invariant ring.

If $k=n-1$, then Example \ref{ex:EasyExample} shows that $A^T$ is not AS regular. (Note that our proposed method of proof does not work here since, in this case, a minimal generating set has three elements.)

Now suppose that $0 \leqslant k \leqslant n-2$ and, for contradiction, assume that $A^T$ is AS regular. Since $g$ is diagonal, we may assume that a generating set for $A^T$ consists of monomials. Therefore, using \eqref{eqn:gOnAx2}, a generating set must contain $ab$ and $d$, as well as $b^i$, where $i = n/\gcd(n,k+1)$. Now, as in Lemma \ref{lem:NotCommutative}, the element $b^{2k+1} c^{n-k-1}$ is also $T$-invariant; we claim that it cannot be generated by $ab,d,$ and $b^i$. To see this, note that (using the notation of Proposition \ref{prop:PresentationForAx}), our claimed generators have respective bidegrees
\begin{align}
(n,n), \quad (0,2n), \quad (i,0), \label{eqn:GeneratorBidegrees}
\end{align}
while $b^{2k+1} c^{n-k-1}$ has bidegree $(2k+1, 2(n-k-1))$. Moreover, we have
\begin{align*}
0 \leqslant k \leqslant \frac{n-2}{2} \quad &\Rightarrow \quad 1 \leqslant 2k+1 \leqslant n-1, \\
\frac{n-1}{2} \leqslant k \leqslant n-2 \quad &\Rightarrow \quad 2 \leqslant 2(n-k-1) \leqslant n-1.
\end{align*}
Therefore both entries in the bidegree of $b^{2k+1} c^{n-k-1}$ are positive, and one of them is at most $n-1$. However, this such a bidegree cannot be written as a positive integer linear combination of the bidegrees \eqref{eqn:GeneratorBidegrees}. Therefore $ab,d,$ and $b^i$ do not form a generating set for $A^T$, so a minimal generating set for $A^T$ contains at least four elements.
\end{proof}


To give a first sufficient condition for $A^T$ to be AS Gorenstein, we calculate the homological determinant of $g$ on $A^x$: 

\begin{lem} \label{lem:TraceAtTInverse}
We have
\begin{align*}
\Tr_{A^x}(g^m,t^{-1}) = \frac{-\omega^{-m(4k+3)} t^4 (1-\omega^{2m(k+1)} t^{4n-2})}{(1-\omega^{-m(k+1)} t) (1-\omega^{-m(2k+1)} t^{2}) (1-\omega^{m(k+1)} t^{2n-1}) (1-t^{2n})}.
\end{align*}
In particular, $\hdet_{A^x}(g) = \omega^{-(4k+3)}$.
\end{lem}
\begin{proof}
This formula for $\Tr_{A^x}(g^m,t^{-1})$ follows from simply evaluating the formula in Lemma \ref{lem:TraceLemma} at $t^{-1}$. The power series expansion of $\Tr_{A^x}(g,t)$ in $\Bbbk \llbracket t^{-1} \rrbracket$ is therefore given by
\begin{align*}
\Tr_{A^x}(g,t) = -\omega^{-m(4k+3)} t^{-4} + \text{lower order terms},
\end{align*}
and so $\hdet_{A^x}(g) = \omega^{-(4k+3)}$.
\end{proof}

\begin{cor}
If $4k+3 \equiv 0 \mod n$ (which happens only if $n$ is odd), then $A^T$ is AS Gorenstein.
\end{cor}
\begin{proof}
This follows from Theorem \ref{thm:TrivialHdet}, Corollary \ref{cor:AxIsGorenstein}, and Lemma \ref{lem:TraceAtTInverse}.
\end{proof}

The above result is only sufficient, but not necessary, for $A^T$ to be AS Gorenstein. Indeed, Example \ref{ex:EasyExample} gives an infinite family of examples where $A^T$ is AS Gorenstein but where the action of $g$ on $A^x$ has nontrivial homological determinant.  

We attempt to identify further examples of AS Gorenstein invariant rings $A^T$ as follows. If $T$ were semisimple, then \cite[Theorem 3.6]{kkz} would apply, and tell us that $A^T$ is AS Gorenstein whenever the action had trivial homological determinant. By Lemma \ref{lem:hdet}, this happens if and only if $4k+2 \equiv 0 \mod n$. If $n \equiv 0 \mod 4$ then no value of $k$ satisfies this equation; otherwise we have
\begin{align}
k = \left\{
\begin{array}{cl}
\frac{n-1}{2} & \text{if } n \text{ is odd}, \\
\frac{n-2}{4} \text{ or } \frac{3n-2}{4} & \text{if } n \equiv 2 \mod 4. \end{array}
\right. \label{eqn:4k+2}
\end{align}
We analyse whether $A^T$ is AS Gorenstein in these cases, despite $T$ not being semisimple, in a series of examples.

Before we begin these examples, we set some notation. Given a finite list of positive integers (possibly with repetitions) $S = [s_1, \dots, s_n]$, let $\mathcal{P}_S(d)$ denote the size of the set
\begin{align*}
\Bigl\{ (i_1, \dots, i_n) \;\Big|\; i_k \in \mathbb{Z}_+, \;\;\sum_{k=1}^n i_k s_k = d \Bigr\}.
\end{align*}
We call these \emph{restricted partition functions}. By \cite[(3.16.4)]{generatingfunctionology}, we have
\begin{align}
\sum_{d \geqslant 0} \mathcal{P}_S(d) t^d = \prod_{s \in S} \frac{1}{(1-t^s)}. \label{eqn:PartitionSummation}
\end{align} 

We begin with the case in \eqref{eqn:4k+2} where $n$ is odd.

\begin{example} \label{ex:LongExample}
Suppose that $n$ is odd and let $k=\frac{n-1}{2}$. Then the matrix of $g^2$ (which also generates $\langle g \rangle$ since $n$ is odd) on $A^x$ is 
\begin{align*}
\begin{pmatrix}
\omega^{-1} & 0 & 0 & 0 \\
0 & \omega & 0 & 0 \\
0 & 0 & 1 & 0 \\
0 & 0 & 0 & 1
\end{pmatrix}.
\end{align*}
By Molien's Theorem, we have
\begin{align*}
\hilb_{A^T}(t) = \frac{1}{n} \sum_{m=0}^{n-1} \Tr_{\restr{(g^2)^m}{A^x}}(t) \hspace{-120pt}& \\
&= \frac{1}{n} \sum_{m=0}^{n-1} \frac{1-\omega^{-2m}t^{4n-2}}{(1-\omega^{-m} t^{2n-1})(1-\omega^{m} t) (1-t^2)(1-t^{2n})} \\
&= \frac{1}{(1-t^2)(1-t^{2n})} \left( \frac{1}{n} \sum_{m=0}^{n-1} \frac{1}{(1-\omega^{m} t)(1-\omega^{-m} t^{2n-1})} - \frac{1}{n} \sum_{m=0}^{n-1} \frac{\omega^{-2m} t^{4n-2}}{(1-\omega^{m} t)(1-\omega^{-m} t^{2n-1})} \right).
\end{align*}
The left-hand term in the parentheses is the Hilbert series of a (weighted) $\mathbb{A}_{n-1}$ singularity, and is therefore known the equal
\begin{align}
\frac{1-t^{2n^2}}{(1-t^n)(1-t^{2n})(1-t^{n(2n-1)})}. \label{eqn:FirstHilbert}
\end{align}
Before evaluating the right-hand term in the parentheses, we first recall the well-known fact that
\begin{align*}
\sum_{m=0}^{n-1} \omega^{jm} = \left\{
\begin{array}{cl}
n & \text{if } j \mod n \equiv 0, \\
0 & \text{if } j \mod n \not\equiv 0.
\end{array}
\right.
\end{align*}
We then have
\begin{align*}
\frac{1}{n} \sum_{m=0}^{n-1} \frac{\omega^{-2m} t^{4n-2}}{(1-\omega^{m} t)(1-\omega^{-m} t^{2n-1})} \hspace{-70pt} & \\
&= \frac{1}{n} \sum_{m=0}^{n-1} \omega^{-2m} t^{4n-2} \Bigg(\sum_{i \geqslant 0} \omega^{im} t^i \Bigg) \Bigg(\sum_{j \geqslant 0} \omega^{-jm} t^{j(2n-1)} \Bigg) \\
&= \frac{1}{n} \sum_{m=0}^{n-1} \omega^{-2m} t^{4n-2} \sum_{\ell \geqslant 0} \Bigg(\sum_{i+j(2n-1) = \ell} \omega^{(i-j-2)m} \Bigg) t^\ell \\
&= \sum_{\ell \geqslant 0} \hspace{2pt} \sum_{i+j(2n-1) = \ell} \Bigg(\frac{1}{n} \sum_{m=0}^{n-1} \omega^{(i-j-2)m} \Bigg) t^{\ell+4n-2} \\
&= \sum_{\ell \geqslant 0} \# \big\{ (i,j) \mid i,j \geqslant 0, \;\; i+(2n-1)j=\ell, \;\; i-j \equiv 2 \mod n \big\} t^{\ell+4n-2}.
\end{align*}
We evaluate the size of the set in this summation by splitting it into two disjoint subsets, and performing various substitutions:
\begin{align*}
\hspace{-15pt}\big\{ (i,j) &\mid i,j \geqslant 0, \;\; i+(2n-1)j=\ell, \;\; i-j \equiv 2 \mod n \big\} \\
&= \big\{ (i,j) \mid i > j \geqslant 0, \;\; i+(2n-1)j=\ell, \;\; i-j \equiv 2 \mod n \big\} \\
&\hspace{13pt}\sqcup \big\{ (i,j) \mid j > i \geqslant 0, \;\; i+(2n-1)j=\ell, \;\; i-j \equiv 2 \mod n \big\} \\
&= \big\{ (j,p) \mid j,p \geqslant 0, \;\; p+(2n-1)j=\ell-1, \;\; p \equiv 1 \mod n \big\} \tag{$i=j+1+p$}\\
&\hspace{13pt}\sqcup \big\{ (i,q) \mid i,q \geqslant 0, \;\; 2ni+(2n-1)q=\ell-2n+1, \;\; q \equiv -3 \mod n \big\} \tag{$j=i+1+q$} \\
&= \big\{ (j,r) \mid j,r \geqslant 0, \;\; nr+2nj=\ell-2 \big\} \tag{$p = nr+1$}\\
&\hspace{13pt}\sqcup \big\{ (i,s) \mid i,s \geqslant 0, \;\; 2ni+n(2n-1)s=\ell-n(2n-5)-2 \big\} \tag{$q=ns+n-3$}.
\end{align*}
The size of this set is given by a sum of restricted partition functions:
\begin{align*}
\mathcal{P}_{n,2n}(\ell-2) + \mathcal{P}_{2n,n(2n-1)}\big(\ell-n(2n-5)-2\big).
\end{align*}
By \eqref{eqn:PartitionSummation}, we have
\begin{align*}
\frac{1}{n} \sum_{m=0}^{n-1} \frac{\omega^{-2m} t^{4n-2}}{(1-\omega^{m} t)(1-\omega^{-m} t^{2n-1})} \hspace{-80pt}& \\
&= \sum_{\ell \geqslant 0} \mathcal{P}_{n,2n}(\ell-2)t^{\ell+4n-2} + \sum_{\ell \geqslant 0} \mathcal{P}_{2n,n(2n-1)}\big(\ell-n(2n-5)-2\big) t^{\ell+4n-2} \\
&= \frac{t^{4n}}{(1-t^n)(1-t^{2n})} + \frac{t^{n(2n-1)}}{(1-t^{2n})(1-t^{n(2n-1)})} \\
&= \frac{t^{4n} - t^{n(2n+3)} + t^{n(2n-1)} - t^{2n^2}}{(1-t^n)(1-t^{2n})(1-t^{n(2n-1)})}.
\end{align*}
Combining this with \eqref{eqn:FirstHilbert}, we obtain
\begin{align*}
\hilb_{A^T}(t) &= \frac{1}{(1-t^2)(1-t^{2n})} \Bigg( \frac{1-t^{2n^2}}{(1-t^n)(1-t^{2n})(1-t^{n(2n-1)})} - \frac{t^{4n} - t^{n(2n+3)} + t^{n(2n-1)} - t^{2n^2}}{(1-t^n)(1-t^{2n})(1-t^{n(2n-1)})}\Bigg) \\
&= \frac{(1-t^{4n})(1-t^{n(2n-1)})}{(1-t^2)(1-t^n)(1-t^{2n})^2(1-t^{n(2n-1)})} \\
&= \frac{1-t^{4n}}{(1-t^2)(1-t^n)(1-t^{2n})^2}.
\end{align*}
This Hilbert series satisfies
\begin{align*}
\hilb_{A^T}(t^{-1}) = -t^{n+2} \hilb_{A^T}(t),
\end{align*}
so Stanley's Theorem tells us that $A^T$ is AS Gorenstein.

The form of the Hilbert series suggests that $A^T$ is generated by four elements, of degrees $2,n,2n,2n$, and that there is a single relation between them in degree $4n$. It is clear that the elements
\begin{align*}
a^n, \quad x \coloneqq ab, \quad y \coloneqq b^n, \quad c, \quad d,
\end{align*}
are all $g$-invariant. Moreover, we have
\begin{align*}
a^n = a \cdot (a^2)^{(n-1)/2} = a \cdot (b^{2(n-1)} d)^{(n-1)/2} = x y^{n-2} d^{(n-1)/2},
\end{align*}
from which it follows that $x,y,c,d$ generate $A^T$. It is straightforward to check that these elements satisfy the relations
\begin{align*}
yx = xy + c^n y, \quad xc = cx, \quad yc = cy,
\end{align*}
and that there is also a relation in degree $4n$, namely
\begin{align*}
x^2 = y^2 d - x c^n.
\end{align*}
It follows that $A^T$ is isomorphic to the quotient of a suitable AS regular algebra by a hyperplane relation.
\end{example}

\begin{example}
We now turn our attention to the case where $n \equiv 2 \mod 4$ and $k=\frac{n-2}{4}$ or $k=\frac{3n-2}{4}$. We also make the additional restriction that $n \geqslant 3$. For $n=2$, the analysis differs: the case $k=\frac{3n-2}{4}=1$ is already covered by Example \ref{ex:EasyExample}, while the case $k=\frac{n-2}{4} = 0$ will be treated in the next example.

We note that, in each case, by \eqref{eqn:gOnAx2} the matrix of $g$ on $A^x$ has the form
\begin{align*}
\begin{pmatrix}
\omega^{-(k+1)} & 0 & 0 & 0 \\
0 & \omega^{k+1} & 0 & 0 \\
0 & 0 & -1 & 0 \\
0 & 0 & 0 & 1
\end{pmatrix}.
\end{align*}
In particular, the order of $\omega^{k+1}$ will affect properties of the invariant ring.

First suppose that $n \equiv 6 \mod 8$ and $k= \frac{n-2}{4}$, or $n \equiv 2 \mod 8$ and $k= \frac{3n-2}{4}$. In this case, it is easy to check that the greatest common divisor of $n$ and $k+1$ is $2$, and hence $\omega^{k+1}$ has order $n/2$. By a calculation similar to the one in Example \ref{ex:LongExample}, one can use Molien's Theorem to show that
\begin{align*}
\hilb_{A^T}(t) = \frac{1-t^{4n}}{(1-t^4)(1-t^{n/2})(1-t^{2n})^2}.
\end{align*}
From this, one checks that
\begin{align*}
\hilb_{A^T}(t^{-1}) = -t^{\frac{n}{2}+4} \hilb_{A^T}(t),
\end{align*}
and hence Stanley's Theorem tells us that $A^T$ is AS Gorenstein.

Now instead suppose that $n \equiv 2 \mod 8$ and $k= \frac{n-2}{4}$, or $n \equiv 6 \mod 8$ and $k= \frac{3n-2}{4}$. In this case, $n$ and $k+1$ are coprime, so $\omega^{k+1}$ has order $n$. As before, one can calculate
\begin{align*}
\hilb_{A^T}(t) = \frac{1-t^{n+4}-t^{4n}+t^{5n+4}}{(1-t^4)(1-t^{\frac{n}{2}+2})(1-t^n)(1-t^{2n})^2},
\end{align*}
and check that
\begin{align*}
\hilb_{A^T}(t^{-1}) = -t^{\frac{n}{2}+2} \hilb_{A^T}(t).
\end{align*}
Again, Stanley's Theorem tells us that $A^T$ is AS Gorenstein.
\end{example}

\begin{example} \label{ex:n2k0Example}
Finally, we consider the case where $n=2$ and $k=0$, which fits into the framework of the previous example with $k=\frac{n-2}{4}$, but exhibits different behaviour. Here, the matrix of $g$ on $A^x$ has the form
\begin{align*}
\begin{pmatrix}
-1 & 0 & 0 & 0 \\
0 & -1 & 0 & 0 \\
0 & 0 & -1 & 0 \\
0 & 0 & 0 & 1
\end{pmatrix},
\end{align*}
and there are many more generators than in the previous example (essentially due  to the fact that $-1 \equiv 1 \mod 2$.) Another calculation with Molien's Theorem tells us that
\begin{align*}
\hilb_{A^T}(t) = \frac{1-t^6-t^7-2t^8-t^9+2t^{11}+2t^{12}+2t^{13}+t^{14}-t^{15}-t^{16}-t^{17}}{(1-t^2)(1-t^3)(1-t^4)^3(1-t^5)},
\end{align*}
which does not satisfy the condition in Stanley's Theorem, and so $A^T$ is not AS Gorenstein. 
\end{example}

\begin{rem}
The above example is particularly interesting because it gives an example of a non-semisimple Hopf algebra acting with trivial homological determinant on an AS regular algebra for which the invariant ring is not AS Gorenstein. This shows that the semisimple hypothesis is essential in \cite[Theorem 3.3]{gourmet}.
\end{rem}

We summarise the results from the preceding examples below.

\begin{thm} \label{thm:WhenGorenstein}
Suppose that $T$ acts on $A$, and that $\sqrtw$ has order $2n$. Then $A^T$ is AS Gorenstein if $n$ and $k$ satisfy the following relationships:
\begin{enumerate}[{\normalfont (1)},topsep=1pt,itemsep=1pt,leftmargin=30pt]
\item $n \geqslant 2$ and $k = n-1$;
\item $4k+3 \equiv 0 \mod n$ (which happens only if $n$ is odd);
\item $n \geqslant 3$ and $4k+2 \equiv 0 \mod n$.
\end{enumerate}
\end{thm}

Additional computations tell us that there are many cases which are not covered by the above result for which $A^T$ is AS Gorenstein. The following table indicates which values of $n$ and $k$ give rise to AS Gorenstein invariant rings $A^T$, which were determined using Stanley's Theorem. Check marks correspond to parameter values for which $A^T$ is AS Gorenstein; circled check marks correspond to those which are covered by Theorem \ref{thm:WhenGorenstein}.

\begin{table}[h]
\centering
\rowcolors{2}{white}{gray!25}
\begin{tabular}{c|cccccccccccccccc}
\diagbox[width=25pt,height=20pt]{$n$}{$k$} & 0 & 1 & 2 & 3 & 4 & 5 & 6 & 7 & 8 & 9 & 10 & 11 & 12 & 13 & 14 & 15\\ \hline
 $2$ &      & \circled{\chk} &      &      &      &      &      &      &      &      &      &      &      &      &      &      \\
 $3$ & \circled{\chk} & \circled{\chk} & \circled{\chk} &      &      &      &      &      &      &      &      &      &      &      &      &      \\
 $4$ &      &      &      & \circled{\chk} &      &      &      &      &      &      &      &      &      &      &      &      \\
 $5$ &      &      & \circled{\chk} & \circled{\chk} & \circled{\chk} &      &      &      &      &      &      &      &      &      &      &      \\
 $6$ &      & \circled{\chk} &      &      & \circled{\chk} & \circled{\chk} &      &      &      &      &      &      &      &      &      &      \\
 $7$ &      & \circled{\chk} &      & \circled{\chk} &      &      & \circled{\chk} &      &      &      &      &      &      &      &      &      \\
 $8$ &      &      &      &      &      &      &      & \circled{\chk} &      &      &      &      &      &      &      &      \\
 $9$ &      & \chk &      &      & \circled{\chk} & \chk & \circled{\chk} &      & \circled{\chk} &      &      &      &      &      &      &      \\
$10$ &      & \chk & \circled{\chk} &      &      &      &      & \circled{\chk} &      & \circled{\chk} &      &      &      &      &      &      \\
$11$ &      &      & \circled{\chk} &      &      & \circled{\chk} &      &      &      &      & \circled{\chk} &      &      &      &      &      \\
$12$ &      & \chk &      & \chk & \chk &      &      & \chk &      &      &      & \circled{\chk} &      &      &      &      \\
$13$ &      &      &      &      &      &      & \circled{\chk} &      &      & \circled{\chk} &      &      & \circled{\chk} &      &      &      \\
$14$ &      &      &      & \circled{\chk} &      &      &      &      &      &      & \circled{\chk} & \chk &      & \circled{\chk} &      &      \\
$15$ &      &      & \chk & \circled{\chk} & \chk & \chk &      & \circled{\chk} &      &      &      &      & \chk &      & \circled{\chk} &      \\
$16$ &      &      &      &      &      &      &      &      &      &      &      &      &      &      &      & \circled{\chk} \\
\end{tabular}
\end{table}

We draw attention to a few features of this table. In Theorem \ref{thm:WhenGorenstein}, the only sufficient condition for $A^T$ to be AS Gorenstein when $n \equiv 0 \mod 4$ is $k=n-1$. Indeed, when $n$ is $4$ or $8$, these are the only situations in which $A^T$ is AS Gorenstein. However, when $n=12$, there are many values of $k$ for which $A^T$ is AS Gorenstein, and so this behaviour is not linked to being divisible by $4$. The final row of the table suggests that, when $n$ is power of $2$, the only AS Gorenstein invariant ring occurs when $k=n-1$. Indeed, the same behaviour holds for $n=32$. It might be possible to prove this purely combinatorially, using Stanley's Theorem.

Moreover, when $n$ is an odd prime, the table suggests that only three values of $k$ give rise to AS Gorenstein invariant rings, and these are the unique values of $k$ arising from each of (1), (2), and (3) of Theorem \ref{thm:WhenGorenstein}. This behaviour continues up to $n=31$. Again, one might be able to prove that this is always the case using combinatorics.

On the other hand, we do not have an understanding of why the cases corresponding to the uncircled check marks are AS Gorenstein, as Theorem \ref{thm:WhenGorenstein} does not apply to these. It would be interesting to investigate these cases more thoroughly.

\bibliographystyle{myamsalpha}
\bibliography{bibliography}

\end{document}